\documentclass[12pt]{article}

\usepackage{pb-diagram}
\usepackage{lamsarrow}
\usepackage[cmtip,arrow]{xy}
\usepackage{amsmath}
\usepackage{amsfonts}
\usepackage{amssymb}
\usepackage{amsthm}
\usepackage{comment}
\usepackage[T1]{fontenc}
\usepackage[utf8]{inputenc}
\usepackage{authblk}
\usepackage{color}
\usepackage[marginpar]{todo}
\usepackage{hyperref}
\usepackage[flushleft]{threeparttable}
\usepackage[top=2cm, bottom=2cm, left=2cm, right=2cm]{geometry}
\usepackage[unreduced,sc]{optional} 

\newcommand {\tb}{\textbf}
\newcommand {\mb}{\mathbb}
\newcommand {\Z}{\mb Z}
\newcommand {\R}{\mb R}
\newcommand {\F}{\mb F}
\newcommand {\C}{\mb C}
\newcommand {\Q}{\mb Q}
\newcommand {\Hp}{\mb H}

\newcommand {\colim}{\textrm{colim}\ }

\newcommand {\lra}{\longrightarrow}

\newcommand {\ad}{\mathrm{ad}}


\newcounter{parentnumber}
\newtheorem{thm}{Theorem}[section]

\newtheorem{cor}[thm]{Corollary}
\newtheorem{defn}[thm]{Definition}
\newtheorem{prop}[thm]{Proposition}
\newtheorem{lmm}[thm]{Lemma}

\newtheorem{exm}[thm]{Example}
\newtheorem{rem}[thm]{Remark}

\begin{document}

\title{Splitting Madsen-Tillmann spectra I. Twisted transfer maps}

\author{Takuji Kashiwabara
Institut Fourier, CNRS UMR \textup{5582}, Universit\'e de Grenoble I,\\
38402 St Martin d'H\`eres cedex France\\
\textit{takuji.kashiwabara@ujf-grenoble.fr}
\and
Hadi Zare
\thanks{The second author has been supported in part by IPM Grant No. 92550117. He also acknowledges partial support from the University of Tehran.}
\\
School of Mathematics, Statistics, and Computer Science, College of Science,\\
University of Tehran, Tehran, Iran \textup{14174}\\
School of Mathematics, Institute for Research in Fundamental Sciences (IPM), P.O. Box: 19395-5746, Niyavaran, Tehran, Iran\\
\textit{hadi.zare@ut.ac.ir}
}

\maketitle

\date{}

\abstract{
\textcolor{black}{We record various properties of twisted Becker-Gottlieb transfer maps and study their multiplicative properties analogous to Becker-Gottlieb transfer. We show these twisted transfer maps factorise through Becker-Schultz-Mann-Miller-Miller transfer; some of these might be well known. We apply this to show that $BSO(2n+1)_+$ splits off $MTO(2n)$, which after localisation away from $2$, refines to a homotopy equivalence $MTO(2n)\simeq BO(2n)_+$ as well as $MTO(2n+1)\simeq *$ for all $n\geqslant0$. This reduces the  study of $MTO(n)$ to the $2$-localised case.   At the prime $2$ our splitting allows to identify some algebraically independent classes in mod $2$ cohomology of $\Omega^\infty MTO(2n)$. We also show that $BG_+$ splits off $MTK$ for some pairs $(G,K)$ at appropriate set of primes $p$, and investigate the consequences for characteristic classes, including algebraic independence and non-divisibility of some universally defined characteristic classes, generalizing results of Ebert and Randal-Williams.}

\tableofcontents

\section{Introduction and statement of results}
For ${\mathbf K}=O, U, SO, SU, Sp, Pin, \textrm{or }Spin,$ the Madsen-Tillmann spectrum $MT{\mathbf K}(n)$ (\cite{GMTW}) is
defined to be $B{\mathbf K}(n)^{-\gamma_n}$, the Thom spectrum of $-\gamma_n$ where {
$\gamma_{\textcolor{black}{n}} $ is the \hypertarget{gamma}{canonical 
bundle  the classifying space for $n$-dimensional $G$-vector bundles $B{\mathbf K}(n)$}} \textcolor{black}{(see Appendix \ref{sec:rec} for notes on classifying spaces)}; the dimension is understood to be over $\C$ in the cases of $U,SU$, and over $\Hp$ in the case of ${\mathbf K}=Sp$. Since the associated infinite loop space $\Omega^\infty MT{\mathbf K}(n)$ classifies fibre bundles
whose fibres are
homeomorphic to an $n$-dimensional manifold with $G$-structure (see e.\,g.\, \cite{Eb} for a nice account), the cohomology ring
$H^*(\Omega^\infty MT{\mathbf K}(n);R)$ then contains the $R$-characteristic classes for such bundles, where
$R$ is some relevant ring. It is known that $H^*(\Omega^\infty _0MT{\mathbf K}(n);\mathbb{Q})$ is just a free commutative
algebra generated by  $H^{*>0}(MT{\mathbf K}(n);\mathbb{Q})$ \opt{sc}{(see Proposition \ref{hinfrational})}. As the torsion-free quotient of  $H^{*>0}(MT{\mathbf K}(n);\mathbb{Z})$ injects to $H^{*>0}(MT{\mathbf K}(n);\mathbb{Q})$, this gives us a good knowledge of the torsion-free quotient of  $H^{*>0}(MT{\mathbf K}(n);\mathbb{Z})$.  To understand
the remaining torsion part, we need to
{know} $\Z/p$-coefficient case,
{which} seems rather difficult. In fact, 
{for $p=2$}, the only existing computations in the literature are due to Galatius and Randal-Williams; they have shown that there exist short exact sequences of Hopf algebras
$$ H_*(\Omega^\infty_0 MT{\mathbf K}(n);\Z /2)\lra H_*(Q_0B{\mathbf K}(n)_+;\Z /2)\lra H_*(\Omega^\infty_0 MT{\mathbf K}(n-1);\Z /2)$$
where ${\mathbf K}=SO$ with $n=2$ (equivalently with  ${\mathbf K}=U$ and $n=1$) \cite[Theorem 1.3]{G}, and ${\mathbf K}=O$
with $n=1,2$ \cite[Theorem A, Theorem B]{Ra}. Here, $Q$ {denotes }$\Omega^\infty\Sigma^\infty$ and the {subscript} $0$ corresponds to the base point component of the
{associated infinite} loop space. The maps { are induced by maps
in the cofibration of spectra} 
below: \cite[Proposition 3.1]{GMTW} \textcolor{black}{(see also Lemma \ref{cofibreoftransfer})}
$$MT{\mathbf K}(n)\stackrel{\omega}{\lra} B{\mathbf K}(n)_+\stackrel{{t}}{\lra} MT{\mathbf K}(n-1).$$
Here  {$\omega$ is the Thomification of the  inclusion \hyperlink{gamma}{${-\gamma _n}$}$\to(-\gamma _n)\oplus\gamma _n$, and $t$ denotes the Becker-Schultz-Mann-Miller-Miller transfer discussed in Section \ref{BSMMM}}.}
The case for ${\mathbf K}=Spin$ with $n=2$ has been treated in \cite[Therorems 1.2, 1.3, 1.7]{Gspin}, the results don't allow
{such a simple
 description.}

At odd primes, as far as we are aware, aside from some degenerate cases, the only computation is due to Galatius for the case of ${\mathbf K}=U$, $n=1$
\cite[Theorem 1.4]{G}. It is therefore of interest for people working in the field to proceed with further computations, or at least identify nontrivial torsion classes in (co-)homology of $\Omega^\infty MT{\mathbf K}(n)$.
We are interested in splitting these spectra, so that some more familiar pieces could be
identified which consequently tell us about pieces of cohomology rings $H^*(\Omega^\infty MT{\mathbf K}(n);\Z/p)$. We shall use standard
methods of stable homotopy theory, which in this paper is mainly based on using various transfer maps, and Steinberg idempotent
as well as the Whitehead conjecture in a sequel \cite{KZ}.

{Now we summarize our main results.  In many cases, we only sketch them, the detailed statement
can be found in the relevant sections.
}

\textcolor{black}{We begin by recording an observation on the twisted Becker-Gottlieb transfer map which
{are probably}
 known to experts, but we don't know of any published account.}

\begin{thm}\label{twistedtransfer}
\begin{enumerate}
 \item (Theorem \ref{twistedtransfer}) For a fibre bundle $\pi:E\to B$ over a ``nice space'' $B$ with fibre a compact manifold, and a vector bundle
$\zeta$ over $B$, one can construct the ``twisted Becker-Gottlieb transfer $t_\pi^\zeta:B^\zeta\lra E^{\pi^*\zeta}$ enjoying similar
properties as the usual Becker-Gottlieb transfer.  Notably the composition
$$B^\zeta\stackrel{t_\pi^\zeta}{\to}E^{\pi^*\zeta}\stackrel{\pi^{{\zeta}}}{\to} B^\zeta$$
induces multiplication by $\chi(F)$ in ordinary homology.
\item (Theorem \ref{transcomp}) For a compact Lie group $G$ and its closed subgroup  $K$, the twisted Becker-Gottlieb transfer factors through
the Becker-Schultz-Mann-Miller-Miller transfer (see Subsection \ref{BSMMM}).
\end{enumerate}
\end{thm}

\textcolor{black}{Before proceeding further, we fix one \textbf{important terminology}. For spaces, we have two distinct notions of {\it splitting}.  We say, when $X\cong Y\times Z$,  {that} $Y$ splits off $X$ (as a direct factor).  As we identify a space with its suspension spectra, we also say, when $\Sigma ^{\infty}
X\cong \Sigma ^{\infty}Y\vee \Sigma ^{\infty}Z'$, that $Y$ splits off $X$ (as a stable wedge summand).  It is easy to see that the first implies the second. Sometimes, we use the same word splitting for two notions,
{the meaning being} clear from the context.}

\hypertarget{conv1}{}
{\opt{short}{Note that the infinite loop space functor $\Omega ^{\infty}$ commutes with the localisation. Thus as}
\opt{sc}{As} our main applications concern mod $p$ (co)homology for given prime $p$, there is no loss of information
by localising at $p$.  Thus from now on we identify a spectrum $X$ with its localisation at the prime $p$
\opt{sc}{(c.f. subsection \ref{sec:local})} unless
otherwise specified. As we
work mainly in the category of spectra, we also identify a (pointed) space $X$ with its suspension spectrum.}

Thus \opt{sc}{by Lemma \ref{lm:split} we get:} \opt{short}{the above implies:}
\begin{cor}\label{cor:splitgeneral}
 \begin{enumerate}
 \item  Let $F\to E\to B$ be as above. If $\chi(F)$ is prime to $p$,
then\opt{local}{localised at $p$}, $B^\zeta$ splits off $E^{\pi^*\zeta}$.
\item Let $(G,K)$ be as above.  If $\chi(G/K)$ is prime to $p$, then $BG^{\alpha}$ splits off $BK^{\alpha|_K \oplus \ad_K-\ad_G|_K}$.
\end{enumerate}

\end{cor}

\textcolor{black}{Corollary \ref{cor:splitgeneral} (ii) is an important tool in proving some of our main splitting results, upon various choices of $\pi:K\to G$ and $\alpha$. Our results below, provide a list of such examples, where the main task is to identify $BK^{\alpha|_K\oplus \ad_K-\ad_G|_K}$ as a Madsen-Tillmann spectrum.}

\begin{thm}
\begin{enumerate}
 \item (Theorems \ref{thsplitmainallprestated}, \ref{thsplitmainusu}) Let $G$, $K$, $p$ be as in Theorems \ref{thsplitmainallprestated} (i), (ii),  or \ref{thsplitmainusu}.
Then the Madsen-Tillmann spectra $MTK$ splits off $BG_+$.
\item (Lemma \ref{MTO(n)-p-odd}) If the prime $p$ is odd, then we have $$MTO(2n)\simeq BO(2n)_+,\ MTO(2n-1)\simeq *.$$
\end{enumerate}
\end{thm}


Thus we have, at odd primes,
$$MTO(2n)\simeq BSO(2n+1)_+\simeq BO(2n)_+\simeq BO(2n+1)_+\simeq BSp(n)_+,$$
where the
\opt{local}{odd primary} {equivalences $BSO(2n+1)_+\simeq BO(2n)_+\simeq BO(2n+1)_+\simeq BSp(n)_+$ are classic. }



\textcolor{black}{Splitting of a spectrum $E$ into a wedge, say $\bigvee_i E_i$, implies that the infinite loop space $\Omega^\infty E$ decomposes as a product of infinite loop spaces $\Omega^\infty E_i$\opt{sc}{ (see Lemma \ref{splittingloopspaeofwedge})}. Thus,} we have the following:

\begin{cor}\label{infsplit}
Let $(K,G)$ and $p$ be as in one of the above theorems.  Then, {as infinite loop spaces, $\Omega ^{\infty}MTK$ decomposes as a product of $QBG_+$ and
{another factor.}}
\end{cor}


Next, notice that for any \textcolor{black}{pointed} space $X$, we have $\Sigma ^{\infty }(X_+)\cong \Sigma ^{\infty }(X)\vee S^0$. {
Thus if $BG_+$ splits $MTK$,
Then so does $S^0$.}
At the level of infinite loop spaces, this implies that {$QS^0$ splits off $\Omega^\infty MTK$}.
%
This splitting however, can also be obtained by {another method. That is,
the Madsen-Tillmann-Weiss map allows us to split $S^0$ from slightly wider class of Madsen-Tillmann spectra, including $MTSp(n)$'s.
Thus:}


\begin{thm}[(Theorem \ref{splitbyMTW})]\label{splitbyMTWintro}
Suppose there exists a manifold $M$ with $K$-structure.  Then $S^0$ splits off $MTK$ at a prime $p$ if $p$ doesn't divide
$\chi (M)$.

Concrete examples are given in the statement of Theorem \ref{splitbyMTW}.
\end{thm}
We note that 
{by either  method}
the map from $MTK$ to $S^0$ is obtained by the composition
$$MTK\stackrel{\omega}{\to}BK_+\stackrel{c}{\to}S^0,$$
where $\omega$ is the Thomification of the  inclusion $-\gamma\to(-\gamma)\oplus\gamma$,
{$\gamma$ denoting the appropriate universal bundle over $BK$,} $c$ is the ``collapse'' map, that is the map that sends the base point to the base point, all the rest to the other point in $S^0$.

At the relevant primes, {Theorem \ref{splitbyMTW} implies
that ${\pi_*}MTK_{\opt{local}{(p)}}$} contains the $\pi _*(S^0)_{\opt{local}{(p)}}$, the stable homotopy groups of the sphere
as a summand.  
It {also implies} that \textcolor{black}{$H
{^*}(\Omega^\infty_0MTK;\Z/p)$  contains a copy of $H{^*}(Q_0S^0;\Z/p)$ as a tensor factor}.
{Thus all non-trivial characteristic classes in $H{^*}(Q_0S^0;\Z/p)$ are non-trivial in
 $H^*(\Omega^\infty MT{\mathbf K}(n);\Z/p)$.  Thus we can generalize
\cite[Theorem 6.1]{Ra}, or rather \cite[Lemma 6.3]{Ra}, as we are not pulling back the characteristic classes to moduli spaces,
and show:}
%
\begin{cor}[(Corollary \ref{charclass1})]\label{Intro-charclass1}
Let $K$ be as in Corollary \ref{charclass1}. Then the composition
$$ MTK\stackrel{\omega}{\lra} BK_+\stackrel{c}{\lra} S^0\stackrel{i}{\lra} KO,$$
where
$i$ is the unit map, induces an injection in mod $2$ cohomology of infinite loop spaces
$$H^*(\Z \times BO;\Z/2)\hookrightarrow  H^*(\Omega^{\infty }MTK;\Z/2).$$
Thus if we define the  class $\xi _i\in H^*(\Omega^{\infty }_0MTK;\Z/2)$ by
$$\xi _i=(\omega \circ c \circ i)^*(w_i), $$ then
they are algebraically independent.

\end{cor}
Let  $F\lra E\stackrel{\pi}{\lra} B$  be a manifold (with suitable structure) bundle, with the Madsen-Tillmann-Weiss map
$f_{\textcolor{black}{\pi}}:B\rightarrow \Omega ^{\infty} _0MTK$.  One can define the characteristic class $\xi _i(E)$ of this bundle simply as the pull-back
$\xi _i(E)=f^*_{\textcolor{black}{\pi}}(\xi _i)$.
Note that as in \cite[Theorem 6.2]{Ra}, one can give a more geometrical interpretation of these characteristic classes, with
the equality $\xi _i(E)=w_i({KO^*(t_f)(1))}$
and
{$KO^*(t_f)(1)$ is the virtual bundle given} by $\Sigma (-1)^i[H^i(F_b,\R)]$ (\cite[Theorem 6.1]{BS}).

Note that in the case of $MTO(2)$, we have, $\tau (\xi _i) = \chi _i$ where the $\tau $ is the conjugation of the Hopf algebra
$H^*(\Omega ^{\infty} _0MTO(2))$, where $\chi _i$'s are defined in \cite[Theorem C]{Ra}.  This is because $w_i(V)$ and
$w_i(-V)$ are related by the conjugation of the Hopf algebra $H^*(BO)$, and the maps of Hopf algebra respect the conjugation.

{The complex analogue of the above, using the Chern classes, also hold, that is, if we use $KU$ instead of
$KO$ and $c_{i(p-1)}$ instead of $w_i$ in the above to define $\xi _i^{\C }$, then we have:}
\begin{cor}[(Corollary \ref{charclass2})]\label{Intro-charclass2}
Let $K$ and $p$ be as in {Theorem \ref{splitbyMTW}}.
The classes $\xi _i^{\C }\in H^*(\Omega ^{\infty}(MTW;\Z/p)$'s are algebraically independent.
\end{cor}
Again we can interpret the characteristic class $\xi _i$ geometrically as before, using appropriate Chern classes and
$KU$-cohomology instead of Stiefel-Whitney classes and $KO$-cohomology.

Another family of characteristic classes, arising from the cohomology of the classifying space $BG$, are discussed in \cite[Subection 2.4]{Ra}.

\begin{defn}\label{def-univchar}
A universally defined characteristic class is an element in the image of the map
$$H^*(BK;R)\stackrel{{\sigma^\infty}^{{*}}}{\lra} H^*(Q_0(BK_+);R)\stackrel{\omega ^*}{\lra} H^*(\Omega ^{\infty}_0MTK;R).$$
We write $\overline{\nu}_c$ for the image of $c\in H^*(BK;R)$ in $H^*(\Omega ^{\infty}_0MTK;R).$
For a manifold bundle $\textcolor{black}{F\lra E\stackrel{\pi}{\lra}B}$ with $K$ structure \textcolor{black}{ on $F$} classified by the Madsen-Tillmann-Weiss map $\textcolor{black}{f_\pi}:B\to \Omega ^{\infty}_0MTK$, $\overline{\nu}_c(E)$ is defined by $\overline{\nu}_c(E)=\textcolor{black}{f_\pi}^*(\overline{\nu}_c)\in H^*(E;R)$.
\end{defn}

This includes Wahl's $\zeta$ classes, Randal-Williams' $\mu$-classes, and the Miller-Morita-Mumford
$\kappa$ classes, we will come back to this later.
{The arguments as in the proof of  \cite[Theorem 2.4]{Ra} show that
this definition agrees with the usual one.
}The method 
of \cite[Example 2.6]{Ra} gives some relations among them.
Our splitting theorem can be used to show that, for classes arising from the summand
$H^*(BSO(2n+1);\Z /2)\subset H^*(BO(2n);\Z /2)$, there can be no other relations. In
a subsequent work {\cite{KZ}}, we will discuss relations among other classes, and in particular establish a complete set of relations when $n=1$.   In many cases, $H^*(BK;R)$ is a polynomial algebra, and if not, it contains a polynomial algebra generated by a family of
characteristic classes (Theorem \ref{cohoclassic}).  So we will use following conventions for the ease of notation.
 If $I=(i_1,i_2, \cdots i_n)$, and $a_1, \cdots a_n$'s are some cohomology classes indexed by integers, then
$a^I$ will denote the monomial $a_1^{i_1}\cdots a_n^{i_n}$.  In the case of the Stiefel-Whitney classes in the cohomology of
$BSO(n)$ or $BSU(n)$, we simply skip the index $i_1$.  Now we can state the following:

\begin{thm}[(Theorem \ref{univ})] Let  \label{Intro-univ}
$\nu _{I}=\overline{\nu}_{Bj^*(w^I)}$ where $j:O(2n)\rightarrow
SO(2n+1)$ will be defined in Section \ref{transfersplitsection}.
Then the only relations among these classes are the ones generated by
$$\nu _{I}^2=\nu _{2I}.$$
Thus the classes $\nu _{i_2,\cdots ,i_{m+1}}$ with at least one $i_k$ odd are algebraically independent.
\end{thm}

Our method can also be applied to the cohomology with integer coefficient.  That is, if $p^I$ denotes the monomial in
Pontryagin classes, then
\begin{thm}[(Theorem \ref{wahlc})]\label{Intro-wahlc}
 The classes {$\zeta _I=\overline{\nu}_{p^I}$} are not divisible in $H^*(\Omega ^{\infty}_0MTO(2m);\Z)$.
\end{thm}
The case $m=1$, combined with the homological stability theorem of \cite{Wa} is Theorem A of \cite{ER}.

%
{We conclude the paper with}
a `non-theorem'; which tells that computations such as Galatius' and Randal-Williams' were somehow exceptional cases and for an infinite family of Madsen-Tillmann spectra, such a description in terms of short exact sequences is not available. We have the following.
\begin{prop}[(Proposition  \ref{nonexact})]\label{Intro-nonexact}
In many cases (a precise hypothesis is given in Proposition  \ref{nonexact}), the sequence of Hopf algebras
$$ H_*(\Omega^\infty_0MT{\mathbf K}(m+1);\Z/p)\lra H_*(Q_0B{\mathbf K}(m+1)_+;\Z/p)\lra H_*(\Omega^\infty_0 MT{\mathbf K}(m);\Z/p)$$
induced by the cofibration for Madsen-Tillmann spectra (Lemma \ref{cofibreoftransfer}) is not short exact.
\end{prop}

However, in  our subsequent work \cite{KZ}, we will exhibit summands of $MTO(n)$'s for which such exact sequences exist.
{ We simply mention
 that in the case of $MTO(2)$, we will have}
\begin{thm}[(\cite{KZ})]\label{MTO(2)}
Let $D(n)$ be the cofibre of $Sp^{2^{n-1}}S^0\to Sp^{2^n}S^0$ induced by {the $2$-fold diagonal} $X\to X^{\times 2}$ where $Sp^{2^n}S^0$ is the
 $2^n$-th symmetric power of $S^0$. Then, completed at $p=2$, we have
$$MTO(2)\simeq BSO(3)_+\vee \Sigma^{-2}D(2).$$
\end{thm}

\begin{description}
\item[Contribution of this paper ]From technical point of view, our main observation from which most of splitting
 results {that we formulated towards the end of this project} will follow and is Theorem \ref{twistedtransfer}; parts of this theorem in special cases
 might be well known, but we don't know any reference{ in their full generality}. In particular, our observations on multiplicative properties of these transfer
 maps recorded as Lemma \ref{multiplicative1} and Theorem \ref{twistedtransfer}(i) (see also Theorem \ref{twistedtranscomp}), as
 well as a push forward formula recorded as Theorem \ref{pushforward}, we believe to be new. Our splitting results are first to appear in
 the literature. Perhaps, from conceptual point of view, one of our important results is the reduction of studying $MTO(n)$ to the
 $2$-primary case, as recorded in Theorem \ref{thsplitmainallprestated}(ii)(see also Lemma \ref{MTO(n)-p-odd}). Similarly, our discussions on
 characteristic classes, their independence, and their integral versions  {generalise previously} known results.
 \item[Organisation of the paper ] We
{start}  by a brief recollection on \opt{sc}{localisation, splitting of spectra, and triviality of the rational stable homotopy. We {then} fix
 some notations on} Thom spectra. Recalling Becker and Gottlieb's variant of Hopf's vector field theorem, we use
 Madsen-Tillmann-Weiss map to prove some splitting results. We then introduce twisted transfer maps, and prove some of their
 multiplicative properties. {Next, we} follow by some factorisation results for twisted transfer maps, as well as
{identification of the} cofibres of some transfer maps using a result of Morisugi. We then use these factorisations to prove more
 splitting results. In the rest of the paper, we study the effect of our results on universally defined characteristic classes.
\item[Notations and conventions ]{
Besides the points made \hyperlink{conv1}{above}, we use the following
conventions.  
We denote by }$X_+$ the space $X$ with the disjoint basepoint added.
\opt{unreduced}{
Thus the notation $H^*(X)$ can mean the (unreduced) cohomology of $X$ as a space, or cohomology of $X$ as a spectrum, which is
the reduced cohomology of  $X$ as a space.  We mainly use it for the latter, thus we use the notation
$H^*(X_+)$ to denote the unreduced cohomology of $X$.  However, in some cases, when there is little risk of confusion,
we write $H^*(X)$ instead.
The only difference is in $H^0$ which is irrelevant to our results in most cases. We use the notation $\widetilde{H}^*(X)$ for the reduced cohomology only when there is
an important risk of confusion.}
 \textcolor{black}{We use the bold letter $\mathbf{K}$  to denote a generic family of Lie groups, that can be specialised to
$\mathbf{K}(n)$.
For instance for $\mathbf{K}=O$.  we have $\mathbf{K}(n)=O(n)$. On the other hand, the normal letters $K$, $G$ etc. will denote
a particular Lie group.} For a (virtual) vector bundle $\alpha\to B$ over some $CW$-complex $B$, we write $B^\alpha$ for the Thom (spectrum) space of $\alpha$. For a space $B$, we use $\R^k$ and $B\times\R^k$ \textcolor{black}{interchangeably} for the $k$-dimensional trivial vector bundle \textcolor{black}{over $B$} which will be clear from the context; the notation $\R^k$ also denotes the $k$-dimensional Euclidean space as usual. The notation $\simeq$ denotes (local) homotopy equivalence of spectra. By abuse of notation, $\cong$ is used to denote homeomorphism of spaces or isomorphism of algebraic objects which will be clear from the context. We shall write $\Z/p$ for the cyclic group of order $p$, and $\Z_{(p)}$ for $p$-localisation of the ring of integers. $p$ will always denote a (positive) prime integer, and all spaces and spectra are localised at $p$ unless stated otherwise.

\end{description}
\opt{sc}{
\section{\textcolor{black}{Recollections on localisation and splitting of spectra}}
\textcolor{black}{This section is intended to explain the framework in which localisation and splitting interact. As an illustration, we record some well known facts about rational stable homotopy and rational homology of infinite loop spaces. This explains our reason to look at localisation of spectra at a given prime $p$.}
\subsection{Localisation of spectra and spaces}
\label{sec:local}
Most of the time we work at one prime at a time.  Thus we can replace safely a spectrum $E$ with its $p$-localisation $E _{(p)}$, that is, a spectrum such that there is a natural map $l_E:E\rightarrow E _{(p)}$ such that $H_*(l_E;k)$ is an isomorphism for any $p$-local coefficients $k$, and $\pi_*( E _{(p)})$ as well as $H_*(E _{(p)};\Z )$ are $p$-local \cite[Proposition 2.4 and Theorem 3.1]{Bspec}.  We have, notably,
\begin{lmm}{\cite[Proposition 2.4]{Bspec}} \label{Lem:locspec}
 $\pi _*( E _{(p)})\cong \pi _*( E )\otimes \Z _{(p)}$.
\end{lmm}
One can also define localisation with respect to any multiplicatively closed set with similar property. { Notably, one can talk about a localisation ``away from $p$'', or even localisation  ``at $0$'',  in other words rationalisation $E_{\Q}$ with similar properties. H}owever, as is well-known, and as we shall review later, the rational stable homotopy theory is rather trivial, so we are very little concerned with rationalisation.  Thus we {\it don't} adopt a wide-spread convention according to which $0$ is considered as a prime number so that ``localisation at a prime $p$'' includes the rationalisation.

The localisation exists as well in the homotopy category of spaces. A space $X$ is called nilpotent if $\pi _n(X)$ has a finite filtration
such that $\pi _1(X)$ acts trivially on each successive filtration quotient.  Then we have
\begin{lmm}{\cite[Proposition 3.1]{BK}} \label{Prop:loc}
If $X$ is nilpotent, then its localisation $X_{(p)}$ satisfies
$H_*(X_{(p)};\Z )\cong H_*(X;\Z )\otimes \Z _{(p)}\mbox{, }\pi _*(X_{(p)})\cong \pi _*(X)\otimes \Z _{(p)}$.
\end{lmm}

\textcolor{black}{Note that for a space $X$, $\pi _1(\Omega X)$ acts trivially on $\pi _n(\Omega X)$, so any loop space is nilpotent. In particular, an infinite loop space is nilpotent.  Thus by using Lemmas \ref{Lem:locspec} and \ref{Prop:loc} we see that the functor $\Omega ^{\infty}$ commutes with the localisation.  Thus for all purpose of this paper, we can replace safely a spectrum $E$ with its $p$-localisation $E _{(p)}$, which we will do throughout the rest of the paper.}}

\subsection{Splitting of spectra}
For a spectrum $E$, we say $E$ splits if there is a homotopy equivalence $E_1\vee E_2\to E$. In practice, often we start with known spectra $E$, $E_1$ and a map among them going one way or the other, and ask whether if there is such a spectrum $E_2$.  We say that $E_1$ splits off $E$ when this is the case. We note that the existence of such a spectrum $E_2$ is equivalent to the existence of a section.  That is, we have
\begin{lmm}\label{lm:split}
 \begin{enumerate}
 \item Let $i:E_1\rightarrow E $ be a map of spectra.  Then $E_1$ splits off $E$ if and only if there is a map of spectra $r:E\rightarrow E_1$ such that $r\circ i$ is a self homotopy equivalence.
\item Let $r:E\rightarrow E_1$  be a map of spectra.  Then $E_1$ splits off $E$ if and only if there is a map of spectra $i:E_1\rightarrow E $ such that $r\circ i$ is a self homotopy equivalence.
\end{enumerate}
\end{lmm}

\begin{proof}
It suffices to take $E_2$ to be the cofibre of $i$.
\end{proof}
We note that in the category of spectra, the wedge sum is a product, that is, for any spectra $X$, we have
a natural isomorphism
$$[X,E_1\vee E_2]\cong [X,E_1]\times [X,E_2].$$  As the infinite loop space functor $\Omega ^{\infty}$ is
the right adjoint of the infinite suspension $\Sigma ^{\infty}$, $\Omega ^{\infty}$ commutes with the product,
thus we see that
\begin{lmm}\label{splittingloopspaeofwedge}
If the spectrum $E$ splits as $E_1\vee E_2\to E$ then so does the infinite loop space
$\Omega ^{\infty}E$ and we have $\Omega ^{\infty}E\cong \Omega ^{\infty}E_1 \times \Omega ^{\infty}E_2.$
\end{lmm}

\subsection{The rational stable homotopy theory}
As we wrote earlier, the rational stable homotopy theory is well-known to be rather trivial.  More precisely
the homotopy group functor $\pi _*$, or the homology group functor $H_*(-;\Q )$ provides
an equivalence between the category of $\Q$-local spectra and
that of graded vector spaces over $\Q$. One way to see this is that any spectrum is
a module over the sphere spectrum $S^0$, so after rationalisation it becomes a module over $S^0_{\Q}$.  However,
the latter is just $H{\Q}$.  Thus a rationalised spectrum is a module over {$H{\Q}$}.  However, for any (ungraded)
{field} $R$, the category of spectra over $HR$ is equivalent to that of graded modules over $R$.

Therefore, splitting $\Q$-local spectra is equivalent to splitting its homology as a graded vector space. {But we have,} for any spectrum $X$,
$$H_*(\Omega ^{\infty }_0(X);\Q)\cong H_*(\Omega ^{\infty }\overline{X};\Q)\cong H_*(\Omega ^{\infty }(\overline{X}_{\Q});\Q),$$
where $\overline{X}$ denotes the $0$-connective cover of $X$, that is a spectrum $\overline{X}$ characterized by the property
that there exists a map $q:\overline{X}\rightarrow X$ such that
$\pi _i(\overline{X})\stackrel{\pi _i(q)}{\cong}(X)$ for $i>0$ and $\pi _i(\overline{X})\cong 0$ for $i\leq 0$. {Furthermore} $\overline{X}_{\Q}$ splits as a wedge of the Eilenberg-Maclane spectra $\Sigma ^nH\Q$'s, {so} it suffices to determine
the homology of $K(\Q ,n)$'s with $n>0$ to determine $H_*(\Omega ^{\infty }_0(X);\Q)$ as a functor of $H_*(\overline{X};\Q)\cong H_{*>0}(X;\Q)$.  $H^*(K(\Q ,n))$ is known to be the free commutative (in the graded sense) algebra generated by $\Q$ concentrated in degree $n$.  Thus we see
\begin{prop}\label{hinfrational}
Let $X$ be any spectrum.  Then $H_*(\Omega ^{\infty }_0(X);\Q)$ is naturally isomorphic to the free commutative (in the graded sense) algebra generated by $H_{*>0}(X;\Q)$.
\end{prop}
We also note that this follows from \cite[Theorem 4.10]{Ku}.

\section{Thom spectra and the transfers}
\subsection{\textcolor{black}{Recollections on Thom spaces and Thom spectra}}\label{rec:Thom}
\textcolor{black}{Thom spaces and Thom spectra, are the main objects that we study in this paper. The aim of section, is to briefly recall some facts about these spaces/spectra. The material here are standard and we refer the reader to \cite{Rudyak} for further references.}


\begin{description}
 \item[Definition ]
Let $X$ be a space, $\zeta $ a $n$-dimensional (real) vector bundle over \textcolor{black}{$X$ equipped with a Riemannian metric}.  In our applications, $X$ will be a classifying space of some Lie group,
$\zeta $ will be a bundle obtained by its representation, but what follows here will be valid for any vector bundle
over any (good) space. \textcolor{black}{We define the Thom space of $\zeta$ by $X^{\zeta}=D(\zeta)/S(\zeta)$ where $D(\zeta)$ and $S(\zeta)$ are the total spaces of disc and sphere bundles associated to $\zeta$, respectively}.
\item[Thom isomorphism ]\hypertarget{DefTiso}{We say that $\zeta$ is orientable for $H^*(-;k)$ if the Thom isomorphism $H^*(X
\opt{unreduced}{_+})\cong \opt{original}{\opt{original}{\tilde}}{H}^{*+n}(X^{\zeta};k)$ holds.
This is the case  for an orientable (in the usual sense)  for any $k$, and for any vector bundle if $k=\Z /2$.
Similar notion of orientability exists for generalized multiplicative cohomology theory.}
\item[Stability of Thom isomorphism ]The Thom isomorphism is stable in the sense that, if $\zeta '=\zeta \oplus \R ^m$, then the Thom isomorphism for $X^{\zeta '}\cong \Sigma ^mX^{\zeta }$ is given by the composition of the Thom isomorphism for $X^{\zeta }$ and the cohomology suspension isomorphism.
Thus it makes sense to talk about the Thom isomorphism for $X^{\zeta }$ when $\zeta$ is a virtual bundle of the form $\zeta '-\R ^m$ with $\zeta '$ a genuine vector bundle over $X$, and in this case $X^{\zeta }$ is defined to be $\Sigma ^{-m}X^{\zeta '}$.
On the other hand, if $X$ is a finite complex, then $KO^0(X)$ is finite, which implies that any virtual bundle $\zeta $ over $X$ can be written as $\zeta '-\R ^m$ with $\zeta'$ a genuine vector bundle.  Thus
we can define the Thom spectra for any virtual bundle over a finite complex in such a way as Thom isomorphism (if oriented) holds.  For general $X$, by filtering it by its finite subcomplexes, and using the naturality arguments, we see that we can define the Thom spectra so that the Thom isomorphism holds when it should.
\item[Thom diagonal ] \hypertarget{Tddef}{}The Thom isomorphism, when it holds, allows us to consider $\opt{original}{\tilde}{H}^*(X^{\zeta };k)$ as a module over $H^*(X \opt{unreduced}{_+};k)$, free of rank $1$. However, we would like to {do so}
without the orientability hypothesis.  For this purpose, we have the ``generalized cup product'' (\cite[IV.5.36]{Rudyak},\cite[2.0.1]{Dobson}) at hand.  That is, if $\zeta $ is a (genuine) vector bundle
over $X$, then the diagonal $X\rightarrow X\times X$ pulls $\zeta \times {0}$ back to $\zeta$, thus induces a map of Thom spaces
$X^{\zeta}\to (X\times X)^{\zeta \times {0}}\cong X^{\zeta}\wedge X_{+}${,}
{called} {\it Thom diagonal}. 
{The induced map in the cohomology}
$H^*(X \opt{unreduced}{_+};k)\otimes \opt{original}{\tilde}{H}^*(X^{\zeta};k)\to  \opt{original}{\tilde}{H}^*(X^{\zeta};k)$
 is
{}called \textit{generalized cup
product} \textcolor{black}{and turns $H^*(X^\zeta;k)$ into a $H^*(X\opt{unreduced}{_+};k)$-module.}  This construction is ``stable'' in the above sense,  thus can be generalized to virtual bundles.
\item[Functoriality ] 
Suppose $\xi\to E$ and $\zeta\to B$ are (virtual) vector bundles, and there is a map of bundles $\xi\to\zeta$ covering a map $f:E\to B$
which is injective on the fiber.  Thus one can use the metric of $f^*(\zeta)$ to induce one on $\xi$, thus it  induces a map of Thom (spectra) spaces.  if $\xi=f^*\zeta$,  we write $f^\zeta$ for this induced map.
\end{description}

\subsection{\textcolor{black}{Hopf's vector field Theorem: Proof of Theorem \ref{splitbyMTW}}}\label{Hopf}
{We recall}  a variant of Hopf's theorem, or rather its useful corollary  for us,
{ which} implies
the main property of transfer \cite[Theorem 5.5]{BG}. Suppose {that} $G$ is a compact Lie group. By a $G$-module $V$, we
 mean a real finite dimensional representation {$V$ of $G$} equipped with $G$-invariant metric. We shall write $S^V$ for the
 $G$-equivariant sphere obtained by one point compactification of $V$. If $F$ is a compact manifold {with a smooth
 $G$ action}, then there exists a $G$-module $V$ together with a $G$-equivariant embedding $i:F\to V$ \cite[Section 2]{BG}. By
 Pontrjagin-Thom construction, this yields a map $i_!:S^V\to F^{\nu_i}$ \textcolor{black}{which we denote by $c$ for simplicity; $i_!$
 is known as umkehr map associated to $i$ (see Subsection \ref{umkehr}).}  {
The embedding $i$ to equip $F$ with a Riemannian metric obtained from the metric of $V$ allows us to consider a trivialisation $\psi:\nu_i\oplus TF\to i^*TV=F\times V$ and
its Thomification  $F^{\nu_i\oplus TF}\to F_+\wedge S^V$ which we still denote by $\psi$}
Finally, let $\pi:F_+\wedge S^V\to S^V$ be {pinch map. }
By Becker-Gottlieb, the following can be obtained from Hopf's vector field theorem.
\begin{thm}\label{Hopfthm}
\cite[Theorem 2.4]{BG} The composition
$$S^V\stackrel{c}{\to} F^{\nu_i}\stackrel{s}{\to} F^{\nu_i\oplus TF}\stackrel{\psi}{\to} F_+\wedge S^V\stackrel{\pi}{\to} S^V$$
where $s$ is the obvious {map obtained by Thomification of the embedding $\nu_i\to\nu_i\oplus TF$}
, is a map of degree $\chi(F)$ - the Euler characteristic of $F$.
\end{thm}
{A few comments are in order.}
\begin{enumerate}
 \item Since $V$ is finite dimensional, the above composite then provides an unstable map.
\item
{For another $G$-module $W$, the composition $F\to
 V\to V\oplus W$ is also a $G$-equivariant embedding.}
The Hopf Theorem then implies that the composition
$$S^{V\oplus W}\to F^{\nu_i\oplus W}\to F^{(\nu_i\oplus TF)\oplus W}
{\to} F_+\wedge S^{V\oplus W}\stackrel{\pi}{\to} S^{V\oplus W}$$
is a map of degree $\chi(F)$.
\item A real finite dimensional inner product space $W$ can be equipped with the trivial $G$-action, thus can be considered as a
`trivial' $G$-module. \label{emb-stable}
\end{enumerate}
The last observation will turn out to be useful while working with twisted Becker-Gottlieb transfer maps.

{We prove  the following to conclude the subsection.}
\begin{thm}\label{splitbyMTW}
Suppose there exists a manifold $M$ with $K$-structure.  Then $S^0$ splits off $MTK$ at a prime $p$ if $p$ doesn't divide
$\chi (M)$.  In particular,
\begin{enumerate}
\item $S^0$ splits off $MTK$ when ${ K}=O(2n)$, $Pin^{+}(4n)$ or $Pin^{-}(4n+2)$ without localisation involved.
\item $S^0$ splits off $MTK$ when ${ K}=SO(2n)$ \opt{local}{after localized at an odd prime.} {if $p$ is odd}.
\item $S^0$ splits off $MTK$ when $K=U(n)$ or $Sp(n)$ { if $p$ doesn't divide $n+1$}.
\end{enumerate}
\end{thm}
\begin{proof}[of Theorem \ref{splitbyMTW}]
{Suppose that $M$ is a manifold with reduction of the structure bundle to $G$.   Choose an embedding of the $m$-dimensional manifold $M$ in an Euclidean space $\R^{m+k}$, say $i:M\lra\R^{m+k}$.  Then denoting $\nu_i$ the normal bundle of the embedding $i$ which we identify with the tubular neighbourhood of $M$ in $\R^{m+k}$, the Thom-Pontrjagin construction provides a map
$$S^{m+k}\lra M^{\nu_i}=M^{\R^{m+k}-TM}.$$
{This construction yields a stable map $S^0\to M^{-TM}$.} 
Denote by  $f:M\lra BG$ the classifying map of the tangent bundle $TM$ of $M$.
Then the stable normal bundle $-TM$ is a pull-back of $-\gamma$ by $f$ where $\gamma $ is the universal vector bundle over $BG$. Consider the Thomified map $M^{-TM} \lra MTG$. }

The composition $S^0\lra M^{-TM} \lra MTG$ is the Madsen-Tillmann-Weiss map associated to the fibre bundle $M\times pt\lra pt$ with $TM$ having a $G$-structure.
It is then a consequence of Hopf's vector filed theorem, \textcolor{black}{stated above as Theorem \ref{Hopfthm} },
that the composition
$$S^0\lra M^{-TM}\lra MTG\stackrel{{\omega}}{\lra} BG_+\stackrel{{c}}{\lra} S^0$$
where ${\omega}:MTG\lra BG_+$ is \textcolor{black}{obtain{ed} by Thomifying the 
inclusion $-\gamma\to (-\gamma)\oplus\gamma$ of bundles over $BG$} and ${c:}BG_+\lra S^0$ the collapse map, is a map of degree $\chi(M)$. Therefore, if $\chi (M)$ is {prime to} $p$, we obtain a splitting of $S^0$ off $MTG$ using the Madsen-Tillmann-Weiss map. Noting that $\R P^{2n}$ has the tangent bundle with structure group $O(2n)$, that can be lifted to $Pin^{\pm}(2n)$ according to the parity of $n$ and $\chi (\R P^{2n})=1$ , we get (i). Noting that $S^{2n}$ has the tangent bundle with structure group $SO(2n)$ and $\chi (SO(2n))=2$, we get (ii). Finally, noting that $\R P^n$, $\C P^n$ and $\Hp P^n$ have the tangent bundle with structure group $O(n)$, $U(n)$ and
$Sp(n)$ respectively,and $\chi (\C P^n)=\chi (\Hp P^n)=n+1$, we get the (iii).
\end{proof}

\subsection{\textcolor{black}{Twisted Becker-Gottlieb transfer maps}}
{We  record some} properties of certain (twisted) transfer maps among Thom spectra that will be useful later on.
\textcolor{black}{Developing and generalising these maps still is a subject of study (see for example \cite[Sections 1.2,4]{ABG}).} 

\subsubsection{The Becker-Gottlieb transfer}\label{BGtransfer}{Let $B$ be a space that admits a filtration by finite subcomplexes (compact subspaces).
  This, of course,  includes the case where $B$ is a finite
complex.
Let $\pi:E\to B$ be a fibre bundle whose {fibre $F$} is a smooth compact manifold.}
Roughly speaking, the transfer construction of \cite{BG} and \cite{BG2}, uses a `geometric integration along fibres' to generalise Theorem \ref{Hopfthm} as follows. The Becker-Gottlieb transfer provides a stable map $t_\pi:B_+\to E_+$ whose
stable homotopy class depends on {the} homotopy class of $\pi$, such that -

\begin{thm}\label{transcomp}\cite[Theorem 5.5]{BG}
The composition
$$B_+\stackrel{t_\pi}{\to}E_+\stackrel{\pi}{\to} B_+$$
induces multiplication by $\chi(F)$ in $\opt{original}{\widetilde}{H}^*(-;\Lambda)$ for any Abelian group $\Lambda$.
\end{thm}
The above theorem implies that if $\chi(F)$ is not divisible by $p$ then
\opt{local}{localised at $p$,} $B_+$ splits off $E_+$. {Moreover, for a variety of reasons, it is useful to twist the above transfer map with a} (virtual) {vector} bundle $\zeta\to B$ {whose construction we postpone to} Subsection \ref{umkehr}
In this case, one obtains a stable map
$$t_\pi^\zeta:B^\zeta\lra E^{\pi^*\zeta}$$
which we call the twisted Becker-Gottlieb transfer{, and whose stable homotopy class depends on homotopy class of $\pi$ and the bundle isomorphism class of $\zeta$}. We have: 
\begin{thm}\label{twistedtranscomp}
The composition
$$B^\zeta\stackrel{t_\pi^\zeta}{\to}E^{\pi^*\zeta}\stackrel{\pi^{{\zeta}}}{\to} B^\zeta$$
induces multiplication by $\chi(F)$ in  $\opt{original}{\widetilde}{H}^*(-;\Lambda)$ for any Abelian group $\Lambda$ where $\pi^{{\zeta}}$ denotes the induced map among Thom spectra. Consequently, if $\chi(F)$ is
{prime to $p$,}
\opt{local}{localised at $p$,}$B^\zeta$ splits off $E^{\pi^*\zeta}$.
\end{thm}
\begin{proof}
\textcolor{black}{By Lemma \ref{multiplicative1}
$${t_\pi^\zeta}^*( \pi^{\zeta *} (x)\cup y)=x\cup t_{\pi}^*(y)$$
for all $x\in H^*B^{\zeta}$ and $y\in H^*E$ where $\cup$ on the left side of the equation denotes the generalized cup product
{$H^*(E^{\pi^*\zeta})\otimes H^*(E)\to H^*(E^{\pi^*\zeta})$ (c.f. Subsection \ref{rec:Thom}).}
The rest follows by setting $y=1$, and noting that $t_{\pi}^*(1)=\chi(F)$ as in the proof of \cite[Theorem 5.5]{BG}.}
\end{proof}


\begin{exm}\label{splitMGtwisted}
Let $G$ be a compact Lie group, $K\subset G$ a closed subgroup{, $V$ } a (virtual) representation of $G$. Then, \opt{local}{{localised at $p$},} $BG^V$ splits off $BK^{V|{_K}}$ if $p\nmid \chi (G/K)$. {In particular,
{if we denote by $N_G(T)$  the normaliser of a maximal torus $T$}, one has $\chi(G/N_G(T))=1$ \cite[Section 6]{BG}. {Thus}  $BG^V$ splits off $BN_G(T)^{V|_{N_G(T)}}$.}
\end{exm}

\subsection{Multiplicative properties of twisted Becker-Gottlieb transfer}\label{BGconstruction1}
{We record some  multiplicative properties of twisted transfer maps,
analogous to the multiplicative properties of the usual Becker-Gottlieb transfer ( \cite[Sections 3,5]{BG}).
They follow immediately from the construction.}

\begin{enumerate} \item
{Suppose $\pi_i:E_i\to B_i$, $i=1,2$ are fibre bundles as above, with $\zeta_i $ (virtual) vector bundles
over $B_i$.}
Suppose further that we have a map of fibre bundles given by the following commutative square
$$
\begin{diagram}
\node{E_1} \arrow{e,t}{h_E}\arrow{s,t}{\pi_1}\node{E_2} \arrow{s,t}{\pi_2} \\
\node{B_1} \arrow{e,t}{h_B}                  \node{B_2}
\end{diagram}
$$
so that the maps $h_E$ and $h_B$ are covered by bundle maps
{$\pi_1^*\zeta_1\to\pi_2^*\zeta_2$ and $\zeta_1\to\zeta_2$}. This yields a commutative square as
$$
\begin{diagram}
\node{B_1^{\zeta_1}}\arrow{e,t}{h_B}\arrow{s,t}{t^{\zeta_1}_{\pi_1}} \node{B_2^{\zeta_2}} \arrow{s,t}{t^{\zeta_2}_{\pi_2}} \\
\node{E_1^{\pi_1^*\zeta_1}}\arrow{e,t}{h_E}                          \node{E_2^{\pi_2^*\zeta_2}}
\end{diagram}
$$
where we have retained $h_E$ and $h_B$ for the Thomified maps. This is analogous to \cite[(3.2)]{BG}.
\item  Next, note that for a fibre bundle $\pi:E\to B$ and a {CW complex $X$ a}dmitting a filtration by finite subcomplexes (compact subspaces), we may consider the fibre bundle $\pi\times 1_X:E\times X\to B\times X$, as well as the vector bundle $\zeta\times 0\to B\times X$. We then have
$$t^{\zeta\times 0}_{\pi\times 1_X}=t_\pi^\zeta\wedge 1:B^\zeta\wedge X_+\to E^{\pi^*\zeta}\wedge X_+.$$
This generalises to $t_{\pi_1\times\pi_2}^{\zeta_1\times\zeta_2}=t_{\pi_1}^{\zeta_1}\wedge t_{\pi_2}^{\zeta_2}$ as
 \cite[(2.2)]{BM}, but we only use the {special} case of $t^{\zeta\times 0}_{\pi\times 1_X}$.
\item  Finally, for the trivial bundle $\pi:F\to \{0\}$, identifying $\{0\}_+=S^0$, the composition $\pi \circ t_\pi:S^0\to S^0$
 has degree $\chi(F)$ \cite[(3.4)]{BG}. Note that this is just Hopf's vector field theorem \textcolor{black}{\ref{Hopfthm}}.
\end{enumerate}
\textcolor{black}{As an application, properties (1) and (2) can be used to prove a multiplicative formula for the (co)homology of twisted transfer maps. We have the following.}
\begin{lmm}\label{multiplicative1}
Suppose $\pi:E\to B$ and $\zeta\to B$ are as above. For $x\in H^*B^{\zeta}$ and $y\in H^*E$, we have
$${t_\pi^\zeta}^*( {\pi^\zeta}^*(x)\cup y)=x\cup {t_\pi}^*(y).$$
Here $\cup$ \textcolor{black}{on the left is a `generalised' cup product
{$$H^*(E^{\pi^*\zeta})\otimes H^*(E)\to H^*(E^{\pi^*\zeta})$$}
whereas the $\cup$ on the right is the usual cup product induced by the usual diagonal.}
\end{lmm}

\begin{proof}
For a fibre bundle $\pi:E\to B$, and a twisting vector bundle $\zeta\to B$, consider $1_B\times \pi:B\times E\to B\times B$. Note that the diagonal map
$d_B:B\to B\times B$ and $(\pi\times 1_E)d_E$ induce a map of fibre bundles
$$
\begin{diagram}
\node{E} \arrow{e,t}{d_E}\arrow{s,t}{\pi}\node{E\times E}\arrow{e,t}{\pi \times 1_E}\node{B\times E} \arrow{s,t}{1_B \times \pi} \\
\node{B} \arrow[2]{e,t}{d_B}                                                      \node[2]{B\times B.}
\end{diagram}
$$
Noting that the vector bundle $0\to B$ is just the identity map $1_B:B\to B$, we see that the horizontal maps are covered by bundle maps
$\zeta\to\zeta\times 0$ and $\pi^*\zeta\to( 1_B\times \pi)^*(\zeta\times 0)=\zeta\times 0$. Now, the above diagram, upon applying properties (1) and (2), yields a commutative diagram as
$$
\begin{diagram}
\node{B^\zeta}       \arrow[2]{e,t}{d_B}\arrow{s,t}{t_\pi^\zeta}   \node[2]{B^\zeta\wedge B_+}\arrow{s,b}{1_{B^\zeta}\wedge t_\pi=t^{\zeta\times 0}_{1_B\times\pi}}\\
\node{E^{\pi^*\zeta}}\arrow{e,t}{d_E}\node{E^{\pi^*\zeta}\wedge E_+}\arrow{e,t}{{\pi^{\zeta}}\wedge
 1_E}\node{B^{\zeta}\wedge E_+}
\end{diagram}
$$
where the horizontal arrows are the Thomification of the horizontal maps in the previous diagram.
The lemma follows, upon taking (co)homology of the above diagram together, using standard properties of products
{noting that $d_B$ and $d_E$ are the \hyperlink{Tddef}{Thom diagonals}} \ref{rec:Thom}.
\end{proof}



\subsection{Umkehr maps: Factoring transfer maps and Gysin homomorphism}\label{umkehr}
The material here ought to be standard and well known; we include a discussion for future reference. We begin with umkehr maps. Suppose $f:E\to B$ is an embedding of compact closed manifolds {with the normal bundle$\nu_f$. The} Pontrjagin-Thom construction yields a map $f_!:B_+\to E^{\nu_f}$
{called} the umkehr map associated to $f$. Moreover, if $\zeta\to B$ is some (virtual) {vector} bundle, then the above construction {may} be twisted to provide an umkehr map \cite{CohenKlein},
\cite[(4.4)]{BeckerSchultz1}
$$f_!^\zeta:B^\zeta\lra E^{\nu_f\oplus \zeta|_E}.$$
Suppose { this time that} $F\to E\stackrel{\pi}{\to} B$ is a fibre bundle
{such that $F$, $E$ and $B$ are all compact closed finite dimensional  manifolds.}
Since $E$ is compact, by Whitney's embedding theorem, we may choose an embedding $\iota:E\to\R^k$, which allows us to extend $\pi$ to an
embedding $j:E\to B\times\R^k$, say $j=(\pi,\iota)$. The umkehr map for such an embedding is a map
$$\textcolor{black}{j_!}:(B\times\R^k)_+=B^{\R^k}\to E^{\nu_j}.$$
The stable normal bundle does not depend on a particular embedding, hence stablising the embedding $\iota$ by composition with embeddings $\R^k\to\R^{k+l}$, $l>0$, we see that the stable homotopy class of $j_!$ does not depend on a particular choice of an embedding $E\to\R^k$. By the discussion of \cite[Section 4]{BG} (see also \cite[Section 2]{BM}) {the stable homotopy class of the composition
$$\Sigma^kB_+=(B\times\R^k)_+\stackrel{\textcolor{black}{j_!}}{\to} E^{\nu_j}\to E^{{\R^k}}=(E\times\R^k)_+=\Sigma^kE_+,$$
is the same as the stable homotopy class of the Becker-Gottlieb transfer $E_+\stackrel{t_\pi}{\to} B_+$. Here, the map $E^{\nu_j}\to E^{{\R^k}}$ is obtained by the decomposition of the trivial bundle $E\times\R^k\simeq\nu_j\oplus N_\pi E$ where $N_\pi E$ is the orthogonal complement of the tangent bundle of $E$ along the fibre of $\pi$; we can talk about $N_\pi E$ as $E$ is compact and has a Riemannian metric.} Moreover, for a (virtual) bundle $\zeta\to B$, using the axial embedding of bundles $\R^k\to\R^k\oplus \zeta$ of bundles over $B$, the Pontrjagin-Thom construction yields a (stable) map
$$\textcolor{black}{j_!^\zeta}:B^{\R^k\oplus \zeta}\to E^{\pi^*\zeta\oplus \nu_j}.$$
{The umkehr map above, leads {to a stable map that}  we call Boardman transfer associated to $\pi$ twisted by $\zeta$ \cite{Boardmanthesis} whose stable homotopy class depends {only} on homotopy class of $\pi$ as well as isomorphism class of $\zeta$ (here, similar to the untwisted case, {l}etting $k$ {large enough} allows to work with the stable normal bundle which is independent of the embedding).} Similarly, the inclusion of bundles $\nu_j\to j^*(\R^k\times B)$ allows us to consider the composition
$$\Sigma^k B^\zeta=B^{\R^k\oplus \zeta}\stackrel{\textcolor{black}{j_!^\zeta}}{\to} E^{\pi^*\zeta\oplus \nu_j}\to E^{\R^k\oplus \pi^*\zeta}=\Sigma^k E^{\pi^*\zeta}$$
which agrees with the twisted Becker-Gottlieb transfer $t_\pi^\zeta:B^\zeta\to E^{\pi^*\zeta}$; this latter being a mere generalisation of the first
factorisation \textcolor{black}{similar to} \cite[Section 6]{BG}. 
We record this for future reference.
\begin{prop}\label{BGfactorisation}
Suppose $\pi:E\to B$ is a fibre bundle whose fibre $F$ is a compact closed manifold, where $B$
is a space that admits a filtration by compact closed manifolds. Moreover, suppose $\zeta\to B$ is a (virtual) vector bundle. Then the (twisted) Becker-Gottlieb transfer $t_\pi^\zeta$ factors through a some 
umkehr map.
\end{prop}

As an immediate application, the above \textcolor{black}{factorisations} allow to relate cohomology of transfer maps to the well known Gysin homomorphisms (also known as the integration along fibers). {Assum}ing
 that $F$ is $n$-dimensional,
$\nu_j$ is
 $(k-n)$-dimensional.
{suppose that} $\pi$ is $R$-orientable for a ring spectrum $R$, that is, if the fiberwise tangent bundle $T_\pi E$ is
\hyperlink{DefTiso}{$R$-orientable} in the usual sense.
{T}hen
{using the  Thom isomorphism $Th:R^*E^{\nu_j}\to R^{*-k+n}E_+$ and the umkehr map $j_!:(B\times\R^k)_+=B^{\R^k}\to E^{\nu_j}$ 
we can define the Gysin homomorphism $\pi_!$ by the composition} 
$${\pi_!= R^{*}(j_!)\circ Th^{-1} :R^{*-k+n}E_+\to R^{*+k}B_+}$$
(see also \cite[Section 4]{ABG}). {We are using the same notation $( )_{!}$ for Gysin homomorphism and umkehr map following
usual conventions, but this shouldn't cause confusion, as the Gysin homomorphism only appear on this page.}

Let $e=e(T_\pi E)\in R^nE_+$ {denote} the Euler class of $T_\pi E$. We then have the following.
\begin{thm}\label{pushforward}
\cite[Theorem 4.3]{BG}
\begin{enumerate}
\item For the Becker-Gottlieb transfer $t_\pi$ we have $t_\pi^*(x)=\pi_!(x\cup e)$.
\item Suppose $\zeta\to B$ is a vector bundle for which Thom isomorphism in $R$-homology holds. Then, for the twisted Becker-Gottlieb transfer $t_\pi^\zeta:B^\zeta\to E^{\pi^*\zeta}$ we have
$${t_\pi^\zeta}^*(x)=\pi_!^\zeta(x\cup e^\zeta)$$
where {$\pi^\zeta_!= R^{*}(j_!^\zeta)\circ (Th^\zeta)^{-1}:R^{*-k+n}E^{\pi^*\zeta}\to R^{*+k}B^\zeta$}  is the twisted Gysin
homomorphism. Here, $Th^\zeta:R^*E^{\nu_j\oplus \pi^*\zeta}\to R^{*-k+n}E^{\pi^*\zeta}$ is the twisted Thom isomorphism, and
$e^\zeta\in R^{n+\dim\zeta}E^{\pi^*\zeta}$ is the image of the Euler class under the Thom isomorphism $R^nE_+\to
R^{n+\dim\zeta}E^\zeta$.
\end{enumerate}
\end{thm}
Part (ii) of the above theorem, like part (i), follows from the above decomposition of the twisted Becker-Gottlieb transfer $t_\pi^\zeta$ through the umkehr map $t^\zeta$ together with an application of the Thom isomorphism, which we refer to the reader to fill in the details in the same manner as \cite{BG} (by twisting  when necessary). Note that choosing $\zeta=0$ yields part (i).\\

\textcolor{black}{Furthermore}, suppose $\xi\to E$ and $\zeta\to B$ are (virtual) {vector} bundles together with a choice of a (relative) trivialisation/framing $\phi:\xi\oplus \R^k\cong\pi^*\zeta\oplus \nu_j$. {Using the relative framing $\phi$, the twisted umkehr ${j_!^\zeta}:B^{\R^k\oplus \zeta}\to E^{\pi^*\zeta\oplus \nu_j}$ provides us with a (stable) map
$$\textcolor{black}{t^{\zeta,\xi}_{\pi,\phi}:}B^\zeta\lra E^\xi$$
referred to as the transfer map associated to the data $(\xi,\zeta,\pi,\phi)$ \cite[(1.5)]{Milequi}, \cite[(2.1)]{Mil2}}.

\begin{rem}\label{BGinfinitecomplex1}
Let's conclude by noting that the compactness assumption above allows to choose $k<+\infty$ and in fact obtain an unstable map
whose stable homotopy class depends on $\pi$ and the twisting bundle $\zeta$. In general, if $E$ and $B$ admit a filtration by finite
subcomplexes or compact submanifolds then inductively by restricting to each filtration, one obtains a collection of maps which
determine a map of Thom spectra whose stable homotopy class depends on $\pi$ and the twisting bundle $\zeta$.
\end{rem}

\subsection{Transfer for Lie groups}
We are interested in the cases that our fibrations arise from Lie groups. {For a compact Lie group} $G$ and $K$ a closed subgroup, we have a fibre bundle
$$G/K\lra BK\stackrel{{\pi}}{\lra} BG.$$
{The classifying spaces $BK$ and $BG$ are not finite dimensional, but admit filtration by finite compact submanifolds. By the discussion in Subsection \ref{BGtransfer} we may speak of the associated Becker-Gottlieb transfer map $BG_+\to BK_+$ which enjoys main the properties of the transfer such as \cite[Theorem 5.5]{BG} (see also Theorem \ref{twistedtranscomp}) which shows that if $\chi (G/K)$ is {prime to} $p$, then $BG_+$  splits off $BK_+$
\opt{local}{(or splits when localised at $p$)}. Such phenomenon is well known, and in the case where $G$ is finite, has been used extensively to study the stable homotopy type of the classifying space $BG$ (e.\,g.\,\cite{SplitLecture}). The case when $G$ is not finite, is also well-known, and for example, it has been shown \cite[Lemma 1]{Yan} that $BSO(2n+1)_+$ splits off $BO(2n)_+$ (this splitting occurs without localisation) and $BSU(n+1)_+$ splits off $BU(n)_+$ unless $p$ divides {$n+1$}.

\subsubsection{Becker-Schultz-Mann-Miller-Miller transfer}\label{BSMMM}
Let $G$ be a compact Lie group and  $M$ a smooth compact manifold with free $G$ action. Consider $\mathfrak{g}$, the Lie algebra of $G$, with an action of $G$ through the adjoint representation. Let $\mu_G=M\times_G\mathfrak{g}\to M/G$ be the adjoint bundle associated to the fibre bundle $M\to M/G$. 
{Note that we may approximate/filter $EG$ by compact manifolds; we shall write $\ad_G=EG\times_G\mathfrak{g}$ for the vector bundle associated to the adjoint representation of $G$ over $BG$ noting that on compact manifold $M$ approximating $EG$ it restrict to $\mu_G$ over $M$}. By compactness of $M$, we may
assume that $M$ has a Riemannian metric. Note that $G$ acts on the tangent bundle $TM$. By the existence of Riemannian metric on $M$, it appears that there is a decomposition of bundles over $M$ \cite[(3.1)]{BeckerSchultz1}, \cite[Lemma 2.1]{Milequi}
\hypertarget{tmgdec}{$$TM/G\cong \mu_G\oplus T(M/G)$$}.
Now, suppose $K<G$ is a closed subgroup and consider the fibre bundle $\pi:M/K\to M/G$. By the umkehr map construction, for some embedding $j:M/K\to\R^k\times M/G$, and a (virtual) vector bundle $\zeta\to M/G$, we have a transfer map
$$(M/G)^{\R^k\oplus \zeta}\lra (M/K)^{\nu_j\oplus \pi^*\zeta}.$$
Now, we have the following commutative diagram.
$$\begin{diagram}
\node{}  \node{}\node{T(M/K)\oplus \nu _j }\arrow{wsw}\arrow{s}\arrow{ese} \node{} \node{} \\
\node{T(\R ^k \times M/G)}\arrow{s}\arrow[4]{e,..} \node{} \node{M/K}\arrow{wsw,r}{j}\arrow{ese,b}{\pi} \node{}
\node{T(M/G)\oplus \R^k}\arrow{s}\\
\node{\R ^k \times M/G}\arrow[4]{e,t}{\pi _2} \node{} \node{} \node{} \node{M/G}
  \end{diagram}
$$
Here $\pi _2$ denotes the projection to the second factor, and all parallelograms are pull-back squares.  Thus we have
$T(M/K)\oplus \nu_j{\simeq j}^*T(M/G\times\R^k){\simeq}\pi^*(T(M/G)\oplus \R^k)$. Now, by plugging in the
\hyperlink{tmgdec}{above decomposition} for $TM/G$ as well as $TM/K$,  we {get} a relative framing
$$\pi^*\mu_G\oplus \nu_j=\mu_K\oplus \R^k.$$
Hence, replacing $\zeta$ with $\mu_G\oplus \alpha$ for {an} arbitrary virtual bundle $\alpha\to M/G$, together with the above relative framing, we obtain a transfer map {as in \cite[(3.7)]{BeckerSchultz1}, \cite[Section 2]{Milequi}}
$$(M/G)^{\mu_G\oplus \alpha}\lra (M/K)^{\mu_K\oplus \pi^*\alpha}.$$
{ We call it} Becker-Schultz-Mann-Miller-Miller transfer {(BSMMM transfer for short)} associated to $\pi$ and twisted with $\alpha$. If we choose $M$ to be a contractible space with a $G$ action, then by approximating $M$ with compact submanifolds, as discussed in Remark \ref{BGinfinitecomplex1}, we may consider a transfer map
$$BG^{\ad_G\oplus \alpha}\lra BK^{\ad_K\oplus \alpha|_K}$$
where $\alpha|_K=\pi^*\alpha$. It may seem that the approximation by compact submanifolds, the transfer map above may not be uniquely determined. However, the graph construction of Becker and Shultz \cite[Section 6]{BeckerSchultz1} (see also \cite[Theorem 3.4 and Theorem 3.9]{Milequi}) allows to obtain a unique stable map, up to weak homotopy equivalence. We note that if $\alpha$ corresponds to a representation $\phi $ of $G$, then $\alpha|_K$ corresponds to the restriction $\phi|_K$, and that the adjoint bundle corresponds to the adjoint representation. We {have seen} that for a fibre bundle $E\to B$ and
{a} twisting bundle $\alpha\to B$, the twisted Becker-Gottlieb transfer admits a factorisation through a suitable umkehr map. This is also the case when we work with compact Lie groups; the Becker-Gottlieb transfer admits a factorisation through BSMMM transfer. The following has to be well known, but we don't know of any published account.
\begin{prop}\label{transferfactor}
Suppose {$G$  and $K$} are as above. Then 
the Becker-Gottlieb transfer $BG_+\to BK_+$ admits a factorisation {through} the BSMMM transfer as
$$BG_+\lra BK^{\ad_K-{\ad_G}|_K}\to BK_+.$$
Similarly, for arbitrary $\eta\to BG$, the twisted Becker-Gottlieb admits a factorisation as
$$BG^{\eta}\to BK^{\eta|_K\oplus \ad_K-{\ad_G|}_K}\lra BK^{\eta|_K}.$$
\end{prop}
\begin{proof}
First, note that for $K<G$, $\pi:BK\to BG$, and $\alpha\to BG$, the twisted Becker-Gottlieb transfer $t_\pi^\alpha:BG^\alpha\to
 BK^{\pi^*\alpha}$, agrees with the Becker-Gottlieb transfer if we choose $\alpha=0$, i.e. choose $\alpha$ to {the} trivial bundle.
 Second, recall that the twisted Becker-Gottlieb transfer admits a factorisation as
$$\Sigma^k B^\zeta=B^{\R^k\oplus \zeta}\to E^{\pi^*\zeta\oplus \nu_j}\to E^{\R^k\oplus \pi^*\zeta}=\Sigma^k E^{\pi^*\zeta}$$
for any fibre bundle $\pi:E\to B$ over some $B$ admitting a filtration by compact subspaces, where $\zeta\to B$ is some {(virtual) vector bundle. By choosing $M$ a compact manifold on which $G$ acts freely, $\pi:E\to B$ the fibre bundle $\pi:M/K\to M/G$, and replacing $\zeta$ with $\mu_G\oplus\alpha$ and using the relative framing $\pi^*\mu_G\oplus \nu_j=\mu_K\oplus \R^k$, we have a factorisation
$$\Sigma^k (M/G)^{\mu_G\oplus\alpha}=B^{\R^k\oplus \mu_G\oplus\alpha}\to (M/K)^{\alpha|_K\oplus\mu_K}\to (M/K)^{\R^k\oplus{\mu_G}|_K\oplus\alpha|_K}.$$}
{By allowing $M$ to approximate $EG$, hence eventually taking $M=EG$, we obtain a factorisation of the twisted Becker-Gottlieb transfer as
$$BG^{\R^k\oplus \ad_G\oplus \alpha}\to BK^{\R^k\oplus \ad_K\oplus \alpha|_K}\to BK^{\R^k\oplus {\ad_G}|_K\oplus \alpha|_K}$$}
which upon choosing $\alpha=-\ad_G$ yields the first factorisation. For the second factorisation, it follows if we simply replace $\alpha$ by {$\eta -\ad_G$}.
\end{proof}
The above proposition, has the following immediate corollary.

\begin{cor}\label{transfersplit}
Let $G$ be a compact Lie group, $K$ its closed subgroup, such that $\chi (G/K)$ is prime to $p$.  Then, localised at the prime $p$,
$BG_+$ splits off $BK^{(\ad_K-\ad_G|_K)}$.
\end{cor}

\subsection{Cofibre of transfer maps}\label{cofibresec}
Suppose we have a fibration $F\to E\to B${, with $\chi(F)$ prime to  $p$. Then}
 the Becker-Gottlieb transfer $t: B_+\to E_+$ transfer provides a stable splitting of $E_+$ off $B_+$.  The other summand
is just the cofibre of $t$, that is, we have a homotopy equivalence $E_+\simeq B_+\vee C_t$.
A similar statement holds, if we replace the Becker-Gottlieb transfer with a twisted BSMMM transfer.
Morisugi's cofibration \cite[Theorem 1.3]{Morisugi} allows us to identify this cofibre in favorable cases. By setting
$E=EG$ in loc.cit. Theorem 1.3, we obtain:
\begin{thm}\label{Mori}
(\cite[Theorem 1.3]{Morisugi})
{Let $G$ be a compact Lie group, $K$ its closed subgroup, such that
there exists a  $G$-representation $V$ with  $G/K=S(V)$ as $G$-spaces,
where $S(V)$ is the sphere in $V$ with a certain $G$-invariant metric. Let $\alpha$ be a vector bundle over $BG$.
Denote by $\lambda$ the vector bundle over $BG$ induced by the representation $V$,  $EG\times_G V\to BG$. Then, there exists a cofibration of spectra:}
$$BG^{\ad_G\oplus \alpha-\lambda}\lra BG^{\ad_G\oplus \alpha}\stackrel{t_K^G}{\lra}BK^{\ad_K\oplus \alpha|_K}\lra
\textcolor{black}{BG^{\R \oplus \ad_G\oplus \alpha-\lambda}\cong \Sigma BG^{\ad_G\oplus \alpha-\lambda}}$$
\end{thm}
 The following lemma provides an application of the above theorem.
\begin{lmm}\label{cofibreoftransfer}
Let ${\mathbf K}=O,SO,Pin^{+},Pin^{-},Spin,U$, or $Sp$. Then there is a cofibration sequence of  spectra
$$MT{\mathbf K}(n+1)\stackrel{\omega}{\lra} B{\mathbf K}(n+1)_+
\stackrel{t}{\lra} \Sigma^{1-d}MT{\mathbf K}(n)\stackrel{j}{\lra} \Sigma MT{\mathbf K}(n+1)$$
where $d=1$ if ${\mathbf K}=O,SO,Pin^{+},Pin^{-}$ or $Spin$, $d=2$ if ${\mathbf K}=U$ or $SU$, and $d=4$ if ${\mathbf K}=Sp$.
\end{lmm}
\begin{proof}First we deal with the case ${\mathbf K}=O,SO,U$ and $Sp$.
Let $\F =\R$ if ${\mathbf K}=O,SO$, $\C$ if ${\mathbf K}=U$, $\Hp$ if ${\mathbf K}=Sp$.
Thus $\F$ is $d$-dimensional vector space over $\R$.  The group ${\mathbf K}(k)$ admits a canoical representation
$\gamma _k^{\F}$ on $\F ^k$, and the corresponding group action preserves the metric.
Thus the ${\mathbf K}(k)$
action on $\F ^k$ restricts to a transitive ${\mathbf K}(k)$ action on the sphere $S^{V}$ where $V$ is $\F ^k$ viewed as
${\mathbf K}(k)$-space.
Now, set $k=n+1$.  The isotropy
subgroup of any unit vector is isomorphic to  ${\mathbf K}(n)$, so we have ${\mathbf K}(n+1)
{\mathbf K}(n)/\cong S(V)$.  Thus we can apply
Theorem \ref{Mori} with $G={\mathbf K}(n+1), K={\mathbf K}(n), \lambda =\gamma _{n+1}^{\F}$.  Set $\alpha=-\ad_G$ for the twisting bundle. It now remains to identify its restriction $\alpha |K$ or its inverse $\ad_G|K$. We have
$$
   \left(
\begin{array}{rr}
X & \\
 & 1
         \end{array}
\right)
\left(
\begin{array}{rr}
A & B \\
{-}B^{\ast } & D
         \end{array}
\right)
 \left(
\begin{array}{rr}
X ^{-1}& \\
 & 1^{}
         \end{array}
\right)=
\left(
\begin{array}{rr}
XAX^{-1}
 & {}XB \\
{-}({}XB)^{\ast } & D
         \end{array}
\right) ,$$
where $X, A$ are $n\times n$ matrices, $B$ is a $n\times 1$ matrix, and $W,D$ are $1\times 1$ matrices with coefficients
in $\F$.
Furthermore, for the $\left(
\begin{array}{rr}
A & B \\
{-}B^{\ast } & D
         \end{array}
\right)$ to lie in the appropriate Lie algebra, we must have $A^*=-A$, $D^*=-D$.  Thus
$A$ is an element of the Lie algebra of $K$.  The block $XB$ corresponds to the canonical representation
$\gamma _n^{\F}$ whereas the block $D$ corresponds to a trivial representation of appropriate dimension.  $D\in \F$ with
$D^*=-D$, so the dimension is $d-1$.  Thus we see that $ad_G|K=ad_K\oplus \gamma _n^{\F}\oplus \R ^{d-1}$.  This
concludes the proof in the cases considered.

The cases ${\mathbf K}=(S)Pin^{\pm}$ follow from the case ${\mathbf K}=(S)O$ noting that the canonical and adjoint
 representations factors through those of the latter.  Finally, the case ${\mathbf K}=SU$ can be handled as in the above,
noting that $\mathfrak{u}(n)=\mathfrak{su}(n)\oplus \mathfrak{u}(1)$ as $SU(n)$-representation,
the splitting map $ \mathfrak{u}(1)\to \mathfrak{u}(n)$ being given by the diagonal divided by $n$.

We note that when ${\mathbf K}=SO,Spin$ or $SU$ and $n=0$, our definition of $B{\mathbf K}(0)$ makes the sequence
$S^d\rightarrow B{\mathbf K}(n)\rightarrow B{\mathbf K}(n+1)$ a fibration.  Thus we can modify the proof of \cite[Theorem 1.3]{Morisugi}
to fit our case.
\end{proof}
\begin{rem}
We note that the cofibre sequences above coincide with those of \cite[Proposition 3.1]{GMTW} when ${\mathbf K}=O$ or $SO$,
and give rise to the fibration of infinite loop spaces as in \cite[(1.3)]{Gspin} when ${\mathbf K}=Spin$ and $n=1$.
To verify these claims, let's note that both the cofibrations of \cite{GMTW} and Morisugi's cofibration are obtained by appealing to James' cofibration. Recall that, for a vector bundle $p:V\to B$ and a (virtual) vector bundle $W\to B$ over a finite complex $B$, there is a cofibre sequence of spectra, due to James:
$$S(V)^{p^*W}\lra B^W\lra B^{V\oplus W}\lra \Sigma S(V)^{p^*W}$$
where $S(V)\to B$ is the sphere bundle of $V$, and the map $B^W\to B^{V\oplus W}$ is  induced by the embedding $W\to V\oplus W$. Hence, it is standard to see that, the map $MTO(n)\to BO(n)_+$ in \cite[Proposition 3.1]{GMTW} and $BG^{\ad_G\oplus \alpha-\lambda}\to BG^{\ad_G\oplus \alpha}$ in \cite[Theorem 1.3]{Morisugi} are both Thomification of the obvious embedding.
Now an easy application of five lemma show that the two cofibre sequences are isomorphic.
The other case, follows by similar considerations.
\end{rem}

As a non-example, where Morisugi's result does not apply, at least integrally, consider embedding of $K=O(n)$ in ${\mathbf
 G}=SO(n+1)$ by $X\mapsto (\det X)(X\oplus 1)$. We have $G/K=\R P^n$ which cannot be identified as a sphere in some vector space
 as $\pi_1\R P^n\simeq\Z/2$.

\subsection{Cohomology of Madsen-Tillmann spectra}
Let ${\mathbf K}=SU,U$ or $Sp$ and $k$ be an arbitrary {field}, or ${\mathbf K}=O$ and $k$ a field of characteristic $2$, and
$d$ be as in Lemma \ref{cofibreoftransfer}. Then we have the following.
\begin{lmm}\label{cohomologymt}
The cofibration $\Sigma ^{-d}MT{\mathbf K}(n-1)\rightarrow MT{\mathbf K}(n)\rightarrow \Sigma ^{\infty} B{\mathbf K}(n)_+$ gives rise to a short exact sequence in cohomology
$$H^*(\Sigma ^{\infty} B{\mathbf K}(n)_+;k)\rightarrow H^*(MT{\mathbf K}(n);k)\rightarrow H^*(\Sigma ^{-d}MT{\mathbf K}(n-1);k),$$
and dually to a short exact sequence in homology
$$H_*(\Sigma ^{-d}MT{\mathbf K}(n-1);k)\rightarrow H_*(MT{\mathbf K}(n);k)\rightarrow H_*(\Sigma ^{\infty}B{\mathbf K}(n)_+;k).$$
Therefore, as graded $k$-vector spaces, we have isomorphisms
$$H_*(MT{\mathbf K}(n);k)\cong \oplus _{j=0}^n \Sigma ^{-dj}H_*(B{\mathbf K}(j);k),H^*(MT{\mathbf K}(n);k)\cong \oplus _{j=0}^n \Sigma ^{-dj}H^*(B{\mathbf K}(j);k).$$
\end{lmm}
\begin{proof}
We have $H^*(B{\mathbf K}(n);k)\cong k[z_i,\cdots ,z_n]$ where $i=2$ if ${\mathbf K}=SU$ and $i=1$ otherwise, with degree of polynomial generators $z_m$ being equal to $dm$. By the Thom isomorphism, we have $H^*(MT{\mathbf K}(n);k)\cong z_n^{-1}k[z_1,\cdots ,z_n].$  Here the notation means the free $k[z_1,\cdots ,z_n]$ module generated by one element $z_n^{-1}$, and we can consider that this is included in an appropriate localisation of $H^*(B{\mathbf K}(n);k)$.  Similarly we have $H^*(MT{\mathbf K}(n-1);k)\cong z_{n-1}^{-1}k[z_1,\cdots ,z_{n-1}]$. Since the canonical representation $\gamma _n$ of ${\mathbf K}(n)$ pulls back to $\gamma _{n-1}\oplus \R$ over ${\mathbf K}(n-1)$, the map $ H^*(MT{\mathbf K}(n);k)\rightarrow H^*(\Sigma ^{-1}MT{\mathbf K}(n-1);k)$ is given by $z_n^{-1}f(z_1{,}\cdots z_{n-1},z_n)\rightarrow \sigma ^{-1}z_{n-1}^{-1}f(z_1{,}\cdots{,} z_{n-1},0)$ where {$f$ is a polynomial in the
cohomology ring and} $\sigma $ denotes the suspension isomorphism which is surjective{. Thus the kernel is clearly
 isomorphic to $k[z_1,\cdots ,z_n]$ which is isomorphic to $H^*(B{\mathbf K}(n);k)$, which
implies  the exactness in the cohomology. By} dualising we get the result in homology. The last statement follows by induction on $n$.
\end{proof}

\section{Splitting Madsen-Tillmann spectra}\label{transfersplitsection}
In this section we deduce the splitting of Madsen-Tillmann spectra from the general theory of splitting of Thom spectra
First, we have:
\begin{thm}\label{thsplitmainallprestated}
\begin{enumerate}
 \item
 Suppose $(K,G)$ is one of the pairs
$$(O(2n),SO(2n+1)),\ (Pin^{+}(4n),Spin(4n+1)),\ (Pin^{-}(4n+2),Spin(4n+3)).$$
Then $BG_+$ stably splits off $MTK$ with no localisation involved.
\item Let $p$ be an odd prime. Let $(K,G)$ be one of the pairs $(SO(2n),SO(2n+1))$, equivalently $(Spin(2n),Spin(2n+1))$, or $(O(2n),O(2n+1{)})$. Then
\opt{local}{ localized at an odd prime $p$,} 
we have $MTK\simeq BG_+\vee\Sigma MTG$.
Furthermore, 
the splitting of $MTO(n)$ reduces to
$$MTO(2n)\simeq BO(2n)_+,\ MTO(2n-1)\simeq *.$$
\end{enumerate}
\end{thm}
We begin with proving part (i) of the Theorem. Consider the embedding 
$$O(2n)\ni X \mapsto j(X)=(\det X)(X\oplus 1)\in SO(2n+1).$$
 One sees that the fibre of $Bj$ is
$$SO(2n+1)/O(2n)\cong \R P^{2n}$$  with $\chi (\R P^{2n})=1$ (with any coefficient). Furthermore, we have
$$
   \left(
\begin{array}{rr}
X & \\
 & W
         \end{array}
\right)
\left(
\begin{array}{rr}
A & B \\
{-}B^{\ast } & D
         \end{array}
\right)
 \left(
\begin{array}{rr}
X ^{-1}& \\
 & W^{-1}
         \end{array}
\right)=
\left(
\begin{array}{rr}
XAX^{-1}
 & W^{-1}XB \\
{-}(W^{-1}XB)^{\ast } & D
         \end{array}
\right) ,$$
where $X, A$ are $2n\times 2n$ matrices and $W,D$ are $1\times 1$ matrices. Replacing $X$ and $W$ with $det(X)\cdot X$ and $det(X)$ respectively, we see that
$$\textcolor{black}{ad_{SO(2n+1)}|_{O(2n)}=}j^*ad_{SO(2n+1)}=ad_{O(2n)}\oplus \gamma _{2n}
{ \mbox{ i.e., }
- \gamma _{2n}=ad_{O(2n)}-ad_{SO(2n+1)}|_{O(2n)}.}$$
Applying Corollary  \ref{transfersplit} to the embedding $j:O(2n)\to SO(2n+1)$ proves Theorem \ref{thsplitmainallprestated}(i) for the pair $(O(2n),SO(2n+1))$.

To prove the other cases, we will need the following.

\begin{lmm}\label{swpullback}
Let $j:O(2n)\rightarrow SO(2n+1)$ be as above.  Then we have
$$Bj^*(w_2)=w_2+nw_1^2$$
in mod 2 cohomology.
\end{lmm}

\begin{proof}
Consider the following commutative diagram where all unnamed arrows are the obvious inclusions, and $\varphi$ is given by
$\textcolor{black}{\varphi} (a_1,\ldots ,a_{{2n}})=(aa_1,\ldots ,aa_{{2n}}, a)$ with $a=\Pi _{i=1}^{{2n}}a_i$.
$$
\begin{diagram}
 \node{O(1)^{2n}}\arrow{s}\arrow[2]{e,t}{\varphi}\node[2]{O(1)^{2n+1}}\arrow{s}\\
\node{O(2n)}\arrow{e,t}{j}\node{SO(2n+1)}\arrow{e}\node{O(2n+1)}
\end{diagram}
$$
Thus to determine $Bj^*(w_2)$, it suffices to compute $B\varphi ^*(\sigma _2(t_1,\ldots ,t_{2n+1}))$ by Theorems \ref{cohoclassic}
and \ref{cohoclassicdetect}.
Now, let's note that in general, we have
$$\begin{array}{lll}
   \sigma _2(\alpha +x_1,\ldots ,\alpha +x_{2n})&=&\Sigma _{1\leq i<k\leq 2n}(\alpha ^2+(x_i+x_k)\alpha+x_ix_k) \\
& = & n(2n-1)\alpha ^2 +  (2n-1)\alpha \sigma _1(x_1,\ldots x_{2n})+\sigma _2(x_1,\ldots x_{2n})
  \end{array}
$$
Thus we get
$$
\begin{array}{lll}
B\varphi ^*(\sigma _2(t_1,\ldots ,t_{2n+1}))& = & B\varphi ^*(\sigma _2(t_1,\ldots ,t_{2n})+t_{2n+1}\Sigma _{i=1}^{2n}t_i)\\
& =&\sigma _2(B\varphi ^*(t_1),\ldots ,B\varphi ^*(t_{2n}))+B\varphi ^*(t_{2n+1})\Sigma _{i=1}^{2n}(B\varphi ^*(t_i)) \\
& = & \sigma _2 ( t +t_1, \ldots , t+t_{2n})+ t \cdot \Sigma _{i=1}^{2n}(t+t_i)\\
& = & n(2n-1)t^2+ t\cdot (2n-1)t +\sigma _2(t_1,\ldots ,t_{2n})+t^2\\
& = & nt^2 +\sigma _2(t_1,\ldots ,t_{2n}) \\
& = & n\sigma _1(t_1,\ldots ,t_{2n})^2+\sigma _2(t_1,\ldots ,t_{2n})
\end{array}
$$
as required, where 
$t=\Sigma _{i=1}^{2n}t_i =\sigma _1(t_1,\ldots ,t_{2n})$.
\end{proof}

\opt{sc}{Now Corollary \ref{pullbackext} implies}
\opt{short}{Noting that $Pin^+()$ groups are classified by $w_2$ and $Pin^{-}()$ groups by $w_2+w_1^2$, we see}
 that $j$ induces a map of double covers $Pin^{\pm}(2n)\rightarrow Spin (2n+1)$, where the sign $\pm$ is $+$ if $n$ is even, $-$ if $n$ is odd. Thus with the choice of appropriate sign, we get the following commutative square
$$\begin{diagram}
   \node{Pin^{\pm}(2n)}\arrow{e,t}{\tilde{j}}\arrow{s} \node{Spin(2n+1)}\arrow{s}\\
\node{O(2n)}\arrow{e,t}{j}\node{SO(2n+1)}
  \end{diagram}
$$
where the vertical arrows are the canonical projection{s}.  Thus we get a diffeomorphism
$$Spin(2n+1)/\tilde{j}(Pin^{\pm}(2n))\cong SO(2n+1)/j(O(2n))\cong \R P^{2n+1}{.}$$
 On the other hand, by definition the canonical representations of $Pin^{\pm}(2n)$ and $Spin(2n+1)$ are the pull-back of the canonical representations of $O(2n)$ and $SO(2n+1)$ by the canonical projection. Furthermore, the adjoint representations of $Pin^{\pm}(2n)$ and $Spin(2n+1)$ are the pull-back of the adjoint representations of $O(2n)$ and $SO(2n+1)$ by the canonical projection, \textcolor{black}{s}ince the kernel of the canonical projection is the center.  Thus we can apply Proposition \ref{transfersplit} to prove Theorem \ref{thsplitmainallprestated}(i) for the other two pairs.

Next, we prove Theorem \ref{thsplitmainallprestated}(ii).

\begin{proof}[of Theorem \ref{thsplitmainallprestated}(ii)]
Th{r}ough the usual embeddings $O(2n)\subset O(2n+1)$ and $SO(2n)\subset SO(2n+1)$ we have diffeomorphisms $$O(2n+1)/O(2n)\cong SO(2n+1)/SO(2n)\cong S^{2n}.$$ Moreover, by passing to the $\Z/2$-central extension, we see that $$Spin(2n+1)/Spin(2n)\cong S^{2n}.$$ Since $\chi(S^{2n})=2$ and {$p$ is odd,  by Corollary}  \ref{transfersplit} the transfer map $BG_+\to BK^{(\ad_K-\ad_G|_K)}$ is split.
The cofibre of transfer maps associated to these embeddings is identified in Lemma \ref{cofibreoftransfer}. The result then follows by the discussion of Section \ref{cofibresec}. The identification of $MTO(n)$ at odd primes is
{postponed} to the end of the section  (Lemma \ref{MTO(n)-p-odd}). This completes the proof.
\end{proof}
%
For the unitary and special unitary groups, we need somewhat odd looking condition on $p$, and we have
\begin{thm}\label{thsplitmainusu}
Let $K=U(n)$, $G=SU(n+1)$.  Suppose that $p$ doesn't divide $n+1$.  Then $BG_+$ splits off stably off $MTK$
\opt{local}{, after $p$-localisation}.
\end{thm}
We begin with a lemma.
\begin{lmm}\label{gtensordet}
Let $p\nmid n+1$.  Then the homomorphism $\varphi: A\mapsto {\det (A)A}$ induces a self homotopy equivalence  of $BU(n)$, as well as a homotopy equivalence $$BU(n)^{-\det \otimes \gamma _n}\simeq BU(n)^{-\gamma _n}=MTU(n).$$
\end{lmm}
\begin{proof}
It suffices to show that it induces an automorphism on $H^*(BU(n);\Z/p)$, as $BU(n)$ is of finite type.  Consider the following commutative diagram
$$
\begin{diagram}
 \node{U(1)^n} \arrow{s} \arrow{e,t}{\overline{\varphi}}   \node{U(1)^n} \arrow{s} \\
\node{U(n)}\arrow{e,t}{\varphi} \node{U(n)}
\end{diagram}
$$
where the vertical arrows are the inclusions of the diagonal matrices with entries in $U(1)$, $\overline{\varphi}$
is given by $$\overline{\varphi}(e^{i\theta _1},\cdots ,e^{i\theta _n})=(e^{i(\theta _1+\theta)},\cdots
e^{i(\theta _n+\theta)})\mbox{ where }\theta =\theta _1+\cdots +\theta _n.$$
Now we see that $H^*(B\overline{\varphi})$ on $H^*(BU(1)^n;\Z/p)\cong \Z _{(p)}[x_1,\cdots ,x_n]$ is given by
$$H^*(B\overline{\varphi})(x_i)=x_i+c_1\mbox{ with }c_1=x_1+\cdots +x_n$$.  Thus by restricting to $H^*(BU(n);\Z _{(p)})\cong \Z _{(p)}[c_1,\cdots ,c_n]$, we see that
$$H^*(B\varphi)(c_1)=(1+n)c_1,H^*(B\varphi)(c_i)\equiv c_i \bmod (c_1)\mbox{ for $i>1$.}$$
Thus $H^*(B\varphi)$ is an automorphism if and only if $p$ doesn't divide $n+1$. Now, we note that the pull-back by $\varphi $ of the canonical representation $\gamma _n$ is just $\det \otimes \gamma _n$, so using the same notation for the bundle and representation,  we get a bundle map
$\det \otimes \gamma _n\lra \gamma _n$ over the map $\varphi$, and thus $-\det \otimes \gamma _n\lra -\gamma _n$ as well.  Since $\varphi$ is a homotopy equivalence, we see that the map between the Thom spectra $BU(n)^{-\det \otimes \gamma _n}\lra BU(n)^{-\gamma _n}=MTU(n)$
is also a homotopy equivalence.
\end{proof}

\begin{proof}[of Theorem \ref{thsplitmainusu}]
Consider the embedding $$U(n)\ni X\mapsto X\oplus (\det X)^{-1}\in SU(n+1).$$
The fibre of the map of classifying spaces $BU(n)\lra BSU(n+1)$  is given by the diffeomorphism $SU(n+1)/U(n)\cong \C P^n$. As in the above,
$$
   \left(
\begin{array}{rr}
X & \\
 & W
         \end{array}
\right)
\left(
\begin{array}{rr}
A & B \\
B^{\ast } & D
         \end{array}
\right)
 \left(
\begin{array}{rr}
X ^{-1}& \\
 & W^{-1}
         \end{array}
\right)=
\left(
\begin{array}{rr}
XAX^{-1}
 & W^{-1}XB \\
(W^{-1}XB)^{\ast } & D
         \end{array}
\right) ,$$
where $X, A$ are $m\times m$ matrices and $W,D$ are $1\times 1$ matrices. By setting $W=det(X)^{-1}$ we see that the representation $\ad_G|_K-\ad_K$ is isomorphic to the tensor product (over $\C$) of the canonical representation with the determinant representation. The proof is complete by Lemma \ref{gtensordet}.
\end{proof}
{
We conclude the section by identifying the $MTO(n)$ spectra at odd primes.}
The following
{generalises the known cases of $MTO(1)$ and $MTO(2)$, c.\ f.\ \cite[subsection 5.1]{Ra}.}
\begin{lmm}\label{MTO(n)-p-odd}Suppose that $p$ is odd.
\opt{local}{Localised at $p$,} {F}or all $n\geqslant 0$, there are homotopy equivalences
$$MTO(2n)\simeq BO(2n)_+\simeq BSO(2n+1)_+\simeq BSp(n)_+,\ MTO(2n+1)\simeq *.$$
\end{lmm}
\begin{proof}
The proof is by strong induction{.}
Note that the induction starts since $MTO(0)\simeq S^0\simeq BO(0) _+$. Suppose now we have $MTO(2n)\simeq BO(2n)_+$ and consider the commutative 
{square}
{$$
\begin{diagram}
       \node{BO(2n+1)_+} \arrow{s,r}{{=}}\arrow{e} \node{MTO(2n)}\arrow{s,r}{\omega} \\
                  \node{BO(2n+1)_+} \arrow{e,t}{t_{B\iota}}               \node{BO(2n)_+}{.}
\end{diagram}
$$
}
{corresponding to the
factorisation of the Becker-Gottlieb transfer through the BSMMM transfer Proposition \ref{transfersplit}.
By the induction hypothesis, the right vertical arrow is a homotopy equivalence.}
Since the inclusion $\iota : O(2n)\subset
O(2n+1)$ induces $p$-local equivalence $B\iota: BO(2n)\simeq BO(2n+1)$ by {\cite[Theorem 1.6]{Thomas}}, and the composition
$$BO(2n+1)_+\stackrel{t_{B\iota}}{\to} BO(2n)_+ \stackrel{{B\iota}}{\to} BO(2n+1)_+$$
{induces} the multiplication by $\chi (S^{2n})=2$ in homology (Theorem \ref{transcomp}), we see that the
{bottom} horizontal
arrow is a homotopy equivalence.
{Thus the top horizontal row is also a homotopy equivalence.  However, by Lemma \ref{cofibreoftransfer}, $MTO(2n+1)$ is its fibre,
thus contractible.}

Next, the cofibration (Lemma \ref{cofibreoftransfer})$$MTO(2(n+1))\to BO(2(n+1))_+
\stackrel{\omega}{\to} MTO(2n+1)$$ together with $MTO(2n+1)\simeq *$ will show that $\omega: MTO(2(n+1))\to BO(2(n+1))_+$ is an equivalence\opt{local}{ at $p$.}{.} This finishes the induction.

The standard maps $BO(2n+1)\to BSO(2n+1)$ and $BSp(n)\to BO(2n)$ are well-known to be
homotopy equivalences. This complete the proof.
\end{proof}
\section{Cohomology of infinite loop spaces associated to the Madsen-Tillmann spectra}
{
The splitting of Madsen-Tillmann spectra discussed
previously (Theorems \ref{splitbyMTW}, \ref{thsplitmainallprestated} and \ref{thsplitmainusu}) implies the splitting of associated
infinite loop space\opt{sc}{ (Lemma \ref{splittingloopspaeofwedge})}.  Thus we can derive some information on their (co)homology, including
information on various characteristic classes that live in their cohomology rings.  In this section we see some examples.}
\subsection{Polynomial families in $H^*(\Omega^\infty MT{G};\Z/p)$}\label{sec:pf}
{We start with the following corollary of Theorem \ref{splitbyMTW}, identifying a polynomial family in the cohomology
ring of $\Omega^{\infty }MTK$ for numerous groups $K$.}
\begin{cor}\label{charclass1}
Let $K$ be $O(2n)$, $U(2n)$, $Sp(2n)$, $Pin^{+}(4n)$ or $Pin^{-}(4n+2)$. The composition
$$ MTK\stackrel{\omega}{\lra} BK_+\stackrel{c}{\lra} S^0\stackrel{i}{\lra} KO,$$
where
$i$ is the unit map, induces an injection in mod $2$ cohomology of infinite loop spaces
$$H^*(\Z \times BO;\Z/2)\hookrightarrow  H^*(\Omega^{\infty }MTK;\Z/2).$$
Thus if we define the  class $\xi _i\in H^*(\Omega^{\infty }_0MTK;\Z/2)$ by
$$\xi _i=(\omega \circ c \circ i)^*(w_i), $$ then we have
$$\Z /2[\xi _1,\ldots ,\xi _k,\ldots ]\subset H^*(\Omega ^{\infty }_0MTG;\Z/2).$$
\end{cor}
\begin{proof}
{  Consider the composition
$$\Omega ^{\infty }MTG \stackrel{\Omega^{\infty}\omega}{\lra} QBG_+\stackrel{Qc}{\lra} QS^0 \lra \Z\times BO.$$
We first show that this composition {induces} an injection in
 cohomology. By Lemma \ref{QS-BO} the map $H_*(QS^0;\Z/2)\to H_*(\Z\times BO;\Z/2)$ is surjective in homology.
By the hypothesis and Corollary \ref{splitbyMTW} the map $\Omega ^{\infty }MTG \lra QBG_+\stackrel{Qc}{\lra} QS^0$ splits, so it is also surjective in homology.  Thus by composing and dualising, we see that $H^*(BO;\Z/2)\cong \Z/2[w_1,\ldots , w_k ,\ldots]$ injects to $H^*(MTG;\Z/2)$.  Noting that the image of $w_k$ is $\xi _k$, we get the desired result.}
\end{proof}

In the special case ${G}=O(2)$, the family discussed above agrees with the one defined in \cite[Section 6]{Ra} up to conjugation, and generates the same subalgebra. 
We now discuss its complex analogue. That is:
\begin{cor}\label{charclass2}
Let $K$ and $p$ be as in \textcolor{black}{Theorem \ref{splitbyMTW}}.
\opt{local}{localied at $p$, t} The composition
$$MTK\stackrel{\omega}{\lra} BK_+\stackrel{c}{\lra} S^0 \stackrel{i}{\lra} KU $$
factors through the Adams summand $E(1)$ and induces an injection in mod $p$ cohomology of infinite loop spaces
$$H^*(\Omega ^{\infty }E(1);\Z/p)\cong \Z/p[c_{p-1},c_{2(p-1)},\ldots ]\hookrightarrow H^*(\Omega ^{\infty} MTK;\Z/p).$$
Thus if we define the  class $\xi _i\in H^*(\Omega^{\infty }_0MTK;\Z/p)$ by
$$\xi _i^{\C}=(\omega \circ c \circ i)^*(c_{i{(p-1)}}), $$
then we have
$$\Z /2[\xi _1^{\C},\ldots ,\xi _k^{\C},\ldots ]\subset H^*(\Omega ^{\infty }_0MTK;\Z/{p}).$$
\end{cor}

\begin{proof}[of Theorem \ref{charclass2}] Note that the unit map of $KU$
\opt{local}{localised at $p$} factors through {that of $E(1)$},
for degree reasons, {as} $\pi _0(\Sigma ^{2i}E(1))=0$ if $0\leq i\leq p-2$.  Thus the result follows by Lemma \ref{cohe1}, Proposition
\ref{splitkono} and Corollary \ref{splitbyMTW} as in the proof of Theorem \ref{charclass1}.
\end{proof}

\subsection{Recollections on homology suspension}
Let $E$ be a spectrum in the sense of \cite{A}, that is, a sequence of pointed spaces $E_j,$ with maps $\Sigma E_j\to E_{j+1}$.
Its homology with coefficients in $k$ is defined to be  $H_*(E;k)=\colim _jH_{*+j}(E_{j};k)$. Note that inside the colimit,
the homology of the basepoint suspending trivially, one can use interchangeably unreduced or reduced homology,
although it is customary to use the reduced homology.  When $E=\Sigma ^{\infty}$
the suspension spectrum of a (pointed) space $X$, we get the isomorphism $H_*(\Sigma ^{\infty }X;k)\cong
\widetilde{H}_*(X;k)$, the unreduced homology of the space $X$. The elementary decomposition of the
direct sum into unreduced homology and the coefficient ring, from our point of view, reflects the splitting of spectra
$\Sigma  ^{\infty }X_+\simeq \Sigma  ^{\infty }X\vee S^0$, and we have
$$H_*(\Sigma  ^{\infty }X_+;k)\cong \widetilde{H}_*(X_+;k) \cong {H}_*(X;k) \cong \widetilde{H}_*(X)\oplus H_*(pt).$$.

The stable homology suspension homomorphism $\sigma_*^\infty:H_*(\Omega^\infty E;k)\to H_*(E;k)$ is
the standard map to the colimit above obtained by replacing $E$ by an equivalent $\Omega$-spectrum.  It can also be defined by
 the map induced by the evaluation map $\Sigma^\infty\Omega^\infty E\to E$ with the latter being adjoint to the identity map $\Omega^\infty E\to \Omega^\infty E$.
{When $X$ is a suspension spectrum, the generalities of adjoint functors ( \cite[Chapter IV, Theorem 1(8)]{Maclanebook}) imply the following:
\begin{lmm}\label{suspensiononto}
For a pointed topological space $X$, the composition
$\Sigma ^{\infty }X\rightarrow \Sigma ^{\infty }QX\rightarrow \Sigma ^{\infty }X$ is the identity, i.e., $\Sigma ^{\infty }X$ splits off $\Sigma ^{\infty }QX$. Thus stable homology suspension
$$\sigma_*^\infty:H_*(QX;k)\to H_*(\Sigma^\infty X;k)\cong \widetilde{H}_*(X;k)$$
is an epimorphism.
\end{lmm}}

However, this is not sufficient for our purpose, since we often have to deal with the map from $H_*(Q_0X;k)$ which is slightly smaller
if $X$ is not connected because of the decomposition
$$QX\simeq Q_0(X)\times \pi _0(QX), \pi _0(QX)\cong \lim \pi _0(\Omega ^n\Sigma ^nX)\cong \pi _0^S(X).$$
Fortunately, the spaces $X$ we deal with have the form $Y_+$ with $Y$ connected.  Thus we have the decomposition
$$QX= QY_+\simeq QY\times QS^0=Q_0Y\times QS^0
.$$ {Noting that $H_*(QS^0)$ suspends to $\widetilde{H}^*(S^0)\cong H_*(pt)$, we deduce} the following:

\begin{lmm}\label{suspensiononto0}
For a connected topological space $Y$, the composition
$$\sigma_*^\infty:H_*(Q_0Y_+;k)\to \widetilde{H}_*(Y_+;k)\to \widetilde{H}_*(Y;k)$$
is onto.
\end{lmm}

It is known that homology suspension kills decomposable elements; this for examples follows from \cite[Corollary 3.4]{Whitehead-elements} applied to the path-loop fibration. An immediate corollary of this observation is the following well-known fact about the homology suspension

\begin{lmm}\label{suspensionkilldec}
Let $k$ be a field. The homology suspension $\sigma_*:H_*(\Omega X;k)\to H_{*+1}(X;k)$ factors through the module of indecomposables (with respect
to the Pontrjagin product) $QH_*(\Omega X;k)$.  In particular, $\sigma_*^\infty:H_*(QX;k)\to H_*(\Sigma^\infty X;k)$ factors through
$QH_*(QX;k)$. Dually, $\sigma ^{\infty *}: \widetilde{H}^*(X;k)\to  H^*(QX;k)$ factors through the set of
primitives $PH^*(QX;k)$.
\end{lmm}
%
Finally, we note that if $f:E\to F$ is a map of spectra then there is a commutative diagram as
$$
\begin{diagram}
\node{H_*(\Omega^\infty E;k)} \arrow{e,t}{\Omega^\infty f}\arrow{s,t}{\sigma_*^\infty} \node{H_*(\Omega^\infty F;k)}
 \arrow{s,t}{\sigma_*^\infty} \\
\node{H_*(E;k)} \arrow{e,t}{f}                                                         \node{H_*(F;k)}
\end{diagram}
$$
that is $\sigma_*^\infty(\Omega^\infty f)_*=f_*\sigma_*^\infty$.
\subsection{The universally defined characteristic classes {in modulo $p$ cohomology}}\label{sec:ucc}
we will now discuss how our splitting results can be used to analyze the universally defined characteristic classes (\ref{def-univchar})
in modulo $p$ cohomology. We will restrict ourselves to the case of $H^*(\Omega ^{\infty}_0 MTO({m});\Z/2)$ for the sake of concreteness.

{The} composition $$H^*(BSO({m}+1) ;\Z/2)
\stackrel{\sigma ^{\infty *}}{\lra }H^*(Q_0BSO({m}+1)_+ ;\Z/2) \lra H^*(\Omega ^{\infty }_0MTO({m}) ;\Z/2)$$ is injective
{because the first map is injective by dualising Lemma \ref{suspensiononto0}, and the second
is so by Corollary \ref{infsplit}.} Of course, it is not a ring map
as $\sigma ^{\infty *}$ is not, {but} its natural right inverse is.  Thus the universally defined characteristic classes that are images of the standard polynomial generators of $H^*(BSO({m}+1) ;\Z/2)$ are algebraically independent.
{Unfortunately} the standard polynomial generators of $H^*(BSO({m}+1);\Z/2)$ do not map to polynomial generators of $H^*(BO({m}) ;\Z/2)$, which makes things a little bit complicated.

For example, let's take the case of $\Omega ^{\infty}_0MTO(2)$, $p=2$.  Then we have
$$H^*(BSO(3) ;\Z/2)\cong \Z/2[w_2,w_3]\mbox{, } H^*(BO(2) ;\Z/2)\cong \Z/2[w_1,w_2],$$ and the map $BO(2)\lra BSO(3)$ induces a map $w_2\mapsto w_2+w_1^2$ by Lemma \ref{swpullback}, and by similar arguments we get  $w_3\mapsto w_1w_2$. One can derive from this the classes  {$\mu _{0,1}+ \mu _{1,0}^2
=\overline{\nu}_{w_2+w_1^2}$ and $\mu _{1,1}=\overline{\nu}_{w_1w_2}$} as defined in \cite{Ra} are algebraically independent.  {A more detailed analysis of the homology suspension map leads to the following:}
\begin{thm} Let $m=2n$ and  \label{univ}
$\nu _I$ be the image of $w^I\in H^*(BSO(m+1);\Z/2)$ in
$H^*(\Omega ^{\infty}_0MTO(m);\Z /2)$ 
under the composition
$$H^*(BSO(m+1);\Z/2)\to H^*(BO(m);\Z/2)\to H^*(Q_0(BO(m));\Z/2)\to H^*(\Omega ^{\infty}_0MTO(m);\Z/2).$$
In other words,
$\nu _I=\overline{\nu }_{Bj^*w^I}$
where $j:O(2n)\rightarrow
SO(2n+1)$
{was}
be defined in Section \ref{transfersplitsection}.
Then the only relations among these classes are the ones generated by
$$\nu _{I}^2=\nu _{2I}.$$
Thus the classes $\nu _{I}$, $I=(i_2,\ldots ,i_{m+1})$ with at least one $i_k$ odd are algebraically independent.
\end{thm}
\begin{proof} Consider the following diagram, which  commutes by the naturality of the homology suspension.
$$
\begin{diagram}
\node{H^*(BSO(2m+1);\Z/2)}\arrow{s,r}{\sigma ^{\infty *}}\arrow{e,t}{Bi^*} \node{H^*(BO(2m);\Z/2)}\arrow{e,t}{\omega ^*}\node{H^*(MTO(2m);\Z/2)}\arrow{s,r}{\sigma ^{\infty *}}\\
\node{H^*(Q_0BSO(2m+1)+;\Z/2)}\arrow[2]{e} \node{}\node{H^*(\Omega ^{\infty}_0MTO(2m);\Z/2)}
\end{diagram}
$$
The bottom horizontal map is injective by Corollary \ref{infsplit}, thus
it suffices to show the corresponding results, with
$\nu _{I}$ replaced with $\sigma ^{\infty *}(w^I)$  in $H^*(Q_0BSO({m}+1_+);\Z/2)$.

Now, let $X$ be any space.
Consider the following diagram.
$$
\begin{diagram}
 \node{H^d(X;\Z/2)}\arrow{e,t}{\sigma ^{\infty *}}\arrow{s,r}{Sq^d= (-)^2} \node{H^d(Q
_0X;\Z/2)}\arrow{e}\arrow{s,r}{Sq^d= (-)^2}
 \node{H^d(X;\Z/2)}\arrow{s,r}{Sq^d= (-)^2} \\
\node{H^{2d}(X;\Z/2)}\arrow{e,t}{\sigma ^{\infty *}} \node{H^{2d}(Q_0X;\Z/2)}\arrow{e} \node{H^{2d}(X;\Z/2)}
\end{diagram}
$$
The maps on the left are induced by map of space, those on the right are induced by those of spectra,
so they commute with the Steenrod squares.
 Moreover, by Lemmata \ref{suspensiononto} and \ref{suspensiononto0} the 
horizontal compositions are the identities.  Therefore, an element of ${{H}}^*(X;\Z/2)$ is a square if and only if its image in $H^*(Q_0X;\Z/2)$ is a square. Furthermore, by Lemma \ref{suspensionkilldec} the map $H^*(X;\Z/2)\to H^*(Q_0X;\Z/2)$ factors through $PH^*(Q_0X;\Z/2)$, and by \cite[Proposition 4.21]{MM}, the only elements in the kernel of the map
$PH^*(Q_0X;\Z/2)\to QH^*(Q_0X;\Z/2)$ are squares.  Thus the kernel of the map $Sym(H^*(X;\Z/2))\to H^*(Q_0X;\Z/2)$ extending the map $H^*(X;\Z/2)\to H^*(Q_0X;\Z/2)$  is the ideal generated by the elements $[x^2]-[x]^2$ where $x\in H^*(X;\Z/2)$, $[x]$ is the corresponding element in $Sym(H^*(X;\Z/2))$.
The proof is complete, once we observe that $H^*(Q_0BSO({m}+1);\Z/2)$ is polynomial. {
However,  $H^*(BSO({m}+1);\Z/2)$ is polynomial, so by  \cite[Theorem 3.11]{W} (see also \cite[Lemma 7.2]{G}) $H^*(Q_0BSO({m}+1);\Z/)2$ is also a polynomial algebra.} 
\end{proof}
\begin{rem}
{\begin{enumerate}
  \item  
We denote $\mu _I=\overline{\nu }_{w^I}$
for $H^*(BO(n);\Z /2)\cong \Z/2[w _1,\ldots ,w _n]$.
This generalises the classes $\mu _{i,j}$'s defined in  \cite[Example 2.6]{Ra}.
It is easy to see that once we express $\nu $'s in terms of $\mu $'s, the relations
$\nu _I^2=\nu _{2I}$
follow from the ones
$\mu _J^2=\mu _{2J}$
and the latter relations were essentially found in \cite{Ra}.
\item The arguments above also apply to other pairs $G,K$ and at any prime
satisfying the hypothesis of Corollary \ref{infsplit}, as long as $H^*(BK;\Z /p)$ is polynomial.
The proof is completely similar{,} 
and at odd primes, we {remark} that we only have to work with the subalgebras
of $H^*(Q_0BK_+;\Z/p)$ generated by the elements dual to that of $H^*(BK;\Z/p)$, which is polynomial by \cite[Theorem 3.11]{W}.
 \end{enumerate}
}
\end{rem}

\begin{exm}
Let's consider the case $n=1$.  As the map induced by $BO(2)\rightarrow BSO(3)$ in cohomology is a ring homomorphism, we have $(w_2)^i\mapsto (w_2+w_1^2)^i$, $w_3^j\mapsto w_1^jw_2^j$.  Thus in low degrees, we have following algebraically independent elements in degrees less than or equal to 9.  We show the detail of computation for first few elements.

\begin{minipage}{\textwidth}
\begin{displaymath}
\begin{array}{|r|l|l|}\hline
degree & \mbox{elements in terms of $\nu$ }& \mbox{elements in terms of $\mu$}\\ \hline
2 & \nu _{1,0} =\overline{\nu}_{w_2}& \mu _{0,1}+\mu _{1,0}^2 =\mu _{0,1}+\mu _{2,0}=\overline{\nu}_{w_2+w_1^2}
\footnotemark \\ \hline
3 &  \nu _{0,1}=\overline{\nu}_{w_3}&\mu _{1,1} =\overline{\nu}_{w_1w_2} \\ \hline
4 & N.A. & N.A. \\ \hline
5 & \nu _{1,1}=\overline{\nu}_{w_2w_3} &\mu _{1,2}+\mu _{3,1}=\overline{\nu}_{w_1w_2^2+w_1^3w_2}  \\ \hline
6 & \nu _{3,0} &\mu _{0,3}+\mu _{1,1}^2+\mu _{4,1}+\mu _{3,0}^2  \\ \hline
7 & \nu _{1,2}& \mu _{2,3}+\mu _{2,1}^2 \\ \hline
8 & \nu _{2,1} & \mu _{2,3}+\mu _{2,1}^2 \\ \hline
9 &\nu _{3,1},\ \nu _{0,3}  &\mu _{1,4}+\mu _{3,3}+\mu _{5,2} +\mu _{7,1},\  \mu _{3,3}\  (resp.)  \\ \hline
  \end{array}
  \end{displaymath}
\end{minipage}

\end{exm}

\subsection{Cohomology with integer coefficients}
In this section, we discuss the implication of our splitting theorems to the cohomology of
{the infinite loop spaces associated to} $MTK$ spectra with $p$-local integer coefficients.
Let $(K,G)$ be a pair satisfying hypotheses of Theorems \ref{thsplitmainallprestated} or \ref{thsplitmainusu}, and choose $p$ accordingly.  Then $p$-locally $BG_+$ splits off $MTK$, therefore $H^*(QBG_+;\Z_{(p)})$ is a \textcolor{black}{tensor factor}
of $H^*(\Omega ^{\infty }MTK;\Z_{(p)})$. Since $H^*(QBG_+;\Z_{(p)})$ can be described completely in terms of
$H^*(BG_+;\Z_{(p)})$, which is completely known in all cases ($H^*(BSpin(n);\Z)$ which we \textcolor{black}{have not discussed in
Appendix \ref{sec:rec}} is known by \cite{KSpin}) we have a complete knowledge of this summand. Unfortunately because of the presence of torsion, we don't have K\"{u}nneth isomorphism,
{so} $H^*(QBG_+;\Z_{(p)})$ as well as  $H^*(\Omega ^{\infty }MTK;\Z_{(p)})$ only have the structure of algebras, and not coalgebras,
which makes it rather difficult to work with them concretely.  However, we still can get some information on them.  For example, it
follows immediately from \cite[Theorem 4.13]{Maybook} that they contain a summand of order $p^i$ for any $i$.

Without localisation, even less can be said.
{Still we can} say, under the hypotheses of Theorem \ref{thsplitmainallprestated}\textcolor{black}{(i)}, we have
$H^*(BG;\Z)$ that splits off $H^*(\Omega ^{\infty }MTK;\Z)$.  As a matter of fact a similar statement holds for any generalized
cohomology. We show that in the case of ordinary cohomology with integer coefficients, this implies that the non-divisibility of generalised
Wahl classes (Theorem \ref{wahlc}). Let's start with a definition.
\begin{defn}
$\zeta _{I}\in H^*(MTO(2m);\Z)$
 is the universally defined characteristic class associated to the monomial in
the Pontryagin classes  $p^I$, $\overline{\nu} _{p^I}$.
\end{defn}
Thus given a $2m$-dimensional {manifold} bundle $E\rightarrow B$ classified by Madsen-Tillmann-Weiss map $f:B\rightarrow \Omega ^{\infty}_0MTO(2m)$,
one can define 
$\zeta _{I}(E)=f^*(\zeta _{I})\in H^*(B;\Z)$. When $m=1$, by writing $i_1=i$, we recover Wahl's classes $\zeta _i$. Given a surface bundle $E\rightarrow B$, Wahl defines $\zeta _i \in H^{4i}(B;\Z)$ to be the image of $p_1(T_v(E))^i$ by the transfer
$H^*(E;\Z )\to H^*((B;\Z)$ where $T_v(E)\to E$ is the vertical tangent bundle(\cite[p.391]{Wa}). Although our definition differs from hers,  as in \cite[Theorem 2.4]{Ra} one can prove that the both definitions agree
\cite[Example 2.5]{Ra}.
{Furthermore, in this case, if the fibre is orientable, then $\zeta _i(E)$ agrees with $\kappa _{2i}(E)$, where $\kappa _{2i}$ is the well-known $2i$-th
Mumford-Miller-Morita class (c.f. loc.cit.).}
\begin{thm}\label{wahlc}
The classes $\zeta _{I}\in H^*(MTO(2m);\Z)$  are not divisible in $H^*(\Omega ^{\infty}_0MTO(2m);\Z)$.
\end{thm}

\begin{proof}
By the naturality of the cohomology suspension, the following square commutes.
$$
\begin{diagram}
\node{H^*(BSO(2m+1);\Z)}\arrow{s}\arrow{e} \node{H^*(BO(2m);\Z)}\arrow{e}\node{H^*(MTO(2m);\Z)}\arrow{s}\\
\node{H^*(Q_0BSO(2m+1);\Z)}\arrow[2]{e} \node{}\node{H^*(\Omega ^{\infty}_0MTO(2m);\Z)}
\end{diagram}
$$
Note that by Theorem \ref{thsplitmainallprestated}, $BSO(2m+1)_+$ splits off $MTO(2m)$, thus $QBSO(2m+1)_+$ splits off $\Omega ^{\infty}MTO(2m)$. By Proposition \ref{suspensiononto} $\Sigma ^{\infty }BSO(2m+1)_+$ splits off $\Sigma ^{\infty }QBSO(2m+1)_+$. Combining these we see that $\Sigma ^{\infty }BSO(2m+1)_+$ splits off $\Sigma ^{\infty }\Omega ^{\infty}MTO(2m)$. Thus the composition $H^*(BSO(2m+1);\Z)\to H^*(\Omega ^{\infty}MTO(2m);\Z)$ is a split monomorphism of abelian groups.

On the other hand, $H^*(BSO(2m+1);\Z)$ is also a direct summand of $H^*(BO(2m);\Z)$, with the quotient group consisting only of torsion elements.  Thus we have a sequence of maps
$$\Z[p_1,\ldots p_m]\subset H^*(BSO(2m+1);\Z) \to H^*(BO(2m);\Z) \cong \Z[p'_1,\ldots p'_m]\oplus T \to \Z[p_1,\ldots p_m]$$
where the composition is an isomorphism.  {Here we used the notation
$p'_i$ to distinguish the Pontryagin classes in $H^*(BO(2m);\Z )$ from  those in $H^*(BSO(2m+1);\Z )$.}
In other words,  a monomial in $p'$'s is, up to torsion elements, the image of a non-divisible element in $H^*(BSO(2m+1);\Z)$. But by definition the $\zeta$-classes are the images of monomials in $p'$'s, thus up to torsion elements, they are images of
non-divisible element in $H^*(BSO(2m+1);\Z)$.  Since $H^*(BSO(2m+1);\Z)\to H^*(\Omega ^{\infty}MTO(2m);\Z)$ is a split mono, a non-divisible element in the former maps to a non-divisible element in the latter. Now, as in the proof of \ref{suspensiononto}, we can replace $ \Omega ^{\infty}MTO(2m)$ with
$\Omega ^{\infty}_0MTO(2m)$ which completes the proof.
\end{proof}

We remark that in general, there is no reason to expect that a monomial in $p'$'s in $H^*(BO(2m);\Z)$
is actually the image of an element in $H^*(BSO(2m+1);\Z)$.  As a matter of fact, Chern classes in $H^*(BSO(2m+1);\Z)$
map to Chern classes in  $H^*(BO(2m);\Z)$, but $c_{2i}$ can pull back to a polynomial involving $c_{2j+1}$ with $j<i$.

{We also note that in the above, we started with non-divisible elements in $H^*(BG;\Z)$.  In some cases,
it may happen that the characteristic class is already divisible in $H^*(BG;\Z)$, in which case its image in
$H^*(\Omega ^{\infty}MTK;\Z)$ will have the same divisibility.  Thus
if we take the pair $(K,G)$ to be $(Pin^{-}(2), Spin(3))$, we get the first part of \cite[Proposition 5.3]{Rpin}, modulo
the homological stability \cite[Theorem 4.19]{Rpin}.}

\subsection{The failure of the exactness}
{We now proceed to prove the following:}
\begin{prop}\label{nonexact}
Suppose the pair of groups $({\mathbf K}(m),G)$, and the prime $p$ satisfies hypotheses of Theorem \ref{thsplitmainallprestated} or Theorem \ref{thsplitmainusu}, so that $BG_+$ splits off $MT{\mathbf K}(m)$, and that $G$ is non-trivial. Suppose further if $K=O$, then $p=2$.  Then, at the prime $p$, the sequence of Hopf algebras
$$ H_*(\Omega^\infty_0MT{\mathbf K}(m+1);\Z/p)\lra H_*(Q_0B{\mathbf K}(m+1)_+;\Z/p)\lra H_*(\Omega^\infty_0 MT{\mathbf K}(m);\Z/p)$$
induced by the cofibration for Madsen-Tillmann spectra (Lemma \ref{cofibreoftransfer}) is not short exact.
\end{prop}
%
\begin{proof}
Let ${\bf K}$, $m$, $G$ and $p$ be as in hypothesis of Proposition{.}
We show that the sequence of Hopf algebras induced by the cofibre sequence $MT{\mathbf K}(m+1)\to \Sigma ^{\infty}B{\mathbf K}(m)_+\to MT{\mathbf K}(m)$
$$H_*(\Omega^\infty _0MT{\mathbf K}(m+1);\Z /p)\lra H_*(Q_0B{\mathbf K}(m+1)_+;\Z /p)\lra H_*(\Omega^\infty _0MT{\mathbf K}(m);\Z /p)$$
is not short exact.  More precisely, we will show that the map
{on the right}
is not surjective. By naturality of the homology suspension, the following square is commutative,
$$
\begin{diagram}
\node{H_*(Q_0B{\mathbf K}(m+1)_+;\Z /p)} \arrow{e}\arrow{s,t}{\sigma_*^\infty} \node{H_*(\Omega^\infty _0MT{\mathbf K}(m);\Z /p)} \arrow{s,t}{\sigma_*^\infty} \\
\node{H_*(\Sigma ^{\infty}B{\mathbf K}(m+1)_+;\Z /p)} \arrow{e}                                                         \node{H_*(MT{\mathbf K}(m);\Z /p).}
\end{diagram}
$$
Suppose that the top horizontal map
 is onto. By  Lemma \ref{suspensiononto0}. the left vertical map is onto. On the other hand, Lemma \ref{cohomologymt} implies that the bottom horizontal map is trivial.  Combining these, we see that the right vertical map
is trivial.  However, our splitting results imply that $H_*(\Omega _0^\infty MT{\mathbf K}(m);\Z /p))$ contains a tensor factor isomorphic to $H_*(Q_0BG_+;\Z /p)$, on which the homology suspension is nontrivial again by Lemma \ref{suspensiononto0}, which is a contradiction.
\end{proof}

\appendix
\section{\texorpdfstring{Recollection on Lie groups, and characteristic classes}{recollection}}
\label{sec:rec}
In this section, we collect some preliminary materials on classical Lie groups, their cohomology, their extension{.}
We mostly intend to fix our notation.

\subsection{Cohomology of classifying spaces and characteristic classes}
\textcolor{black}{For a moment, let's write $\mathrm{Gr}^{\mathbf{G}}(d,+\infty)$ for the classifying space of $\mathbf{G}(d)$-vector bundles. Often, the notation $B\mathbf{G}(d)=E\mathbf{G}(d)/\mathbf{G}(d)$ denotes the classifying space of a $\mathbf{G}(d)$ which most of the time coincides $\mathrm{Gr}^{\mathbf{G}}(d,+\infty)$. However, for some choices of $\mathbf{G}$, in the case of $d=0$, there are a few exceptions. For instance, for $\mathbf{G}=SO$, $\mathrm{Gr}^{\mathbf{SO}}(d,+\infty)=\cup _kG^+(k,d+k)$ where $G^+(k,d+k)$ is the Grassmann manifold of oriented $d$-codimensional linear subspaces of $\R^{d+k}$, yields $\mathrm{Gr}^{\mathbf{G}}(d,+\infty)=S^0$ which is {\it not} homotopy equivalent to $BSO(0)\simeq B1\simeq *$. This occurs because of existence of $+$ and $-$ orientations for a point. Similarly, $\mathrm{Gr}^{\mathbf{Spin}}(0,+\infty)=B\Z/2\times S^0$ and $\mathrm{Gr}^{\mathbf{SU}}(0,+\infty)=S^1$, do not agree with $BSpin(0)$ and $BSU(0)$, respectively. By abuse of notation, we keep writing $B\mathbf{G}(d)$ for the classifying space of $\mathbf{G}(d)$-vector bundles as we declared in Section $1$.}\\
The following is well-known: the ring structure is given by \cite[Proposition 23.2]{Bo53} for ${\mathbf K}=SO$, Theorem 19.1 loc.cit in other cases. The computation for ${\mathbf K}=SO,O$ with integral coefficient is  \cite[Theorem A, Theorem 12.1]{Thomas}.   The identification of generators with characteristic classes follow, for example, from \cite[Section 9]{BH}.

\begin{thm}\label{cohoclassic}
Let $k$ be any ring if ${\mathbf K}=U,Sp$ or $SU$, an algebra over $\Z/2$ if ${\mathbf K}=O$ or $SO$, $k'$ be a ring in which
$2$ is invertible. Then, for $n\geqslant 1$, $H^*(B{\mathbf K}(n);k)$ is given as follows:
for $n\geq 0$ we have
$$
\begin{array}{lll}
H^*(BO(n);k)& \cong & k[w_1,w_2,\ldots ,w_n]\\
H^*(BU(n);k) & \cong &  k[c_1,c_2,\ldots ,c_n]\\
H^*(BSp(n);k) & \cong & k[p_1,p_2,\ldots ,p_n]\\
\end{array}
$$
and for $n\geq 1$ we have
$$
\begin{array}{lll}
H^*(BSO(n);k) & \cong & k[w_2,\ldots ,w_n] \\
H^*(BSU(n);k) & \cong & k[c_2,\ldots ,c_n]
\end{array}
$$
and further for $m\geq 1$ we have
$$
\begin{array}{lll}
H^*(BSO(2m);\Z) & \cong & \Z[p_1,\cdots p_m,\chi]/(\chi ^2-p_m)\oplus T\\
H^*(BSO(2m);k') & \cong & k'[p_1,\cdots p_m,\chi]/(\chi ^2-p_m)\\
H^*(BSO(2m+1);\Z )\cong H^*(BO(2m+1);\Z ) \cong H^*(BO(2m);\Z )& \cong & \Z [p_1,\cdots p_m] \oplus T\\
H^*(BSO(2m+1);k')\cong H^*(BO(2m+1);k') \cong H^*(BO(2m);k')& \cong & k'[p_1,\cdots p_m]
\end{array}
$$
where $w_i$, the $i$-th Stiefel-Whitney class, has degree $i$,  $c_i$, the $i$-th Chern class, has degree $2i$, and
{$p_i\in H^{4i}(BSp(n);k)$}, the $i$-th symplectic Pontryagin class, $p_i\in H^{4i}(BSO(n);k')$, the $i$-th Pontryagin class, $T$ is
an elementary abelian $2$-group.  Furthermore, the standard inclusions $SO(n)\subset O(n)$
 induce the obvious projections sending $w_1$  to $0$ and other $w_i$'s to $w_i$'s
with $k$ coefficients, and similar statement holds for the standard inclusions $SU(n)\subset U(n)$.
The inclusions $O(n)\subset U(n)$, sends $c_i$ to $w_i^2$ when characteristic of $k$ is 2, otherwise $c_{2i}$ to $p_i$.
\end{thm}
This can be stated in a more economic way by saying that for ${\mathbf K}=O,U$ or $Sp$, $H^*(B{\mathbf K(n);k})\cong
k[x_1,\ldots x_n]$ with the degree of $x_i$ equal to $di$, where $d=1,2$ or $4$ depending on whether ${\mathbf K}=O,U$ or $Sp$,
similarly for $H^*(B{\mathbf SG(n);k})$. Then the standard inclusions ${\mathbf K}(n-1)\subset {\mathbf K}(n)$ induce the obvious
projections sending $x_n$ to $0$, and other $x_i$'s to $x_i$.

As the names suggest, these polynomial generators are characteristic classes, more precisely the characteristic classes for universal
bundles, or the universal characteristic classes.  That is, for example, if $V$ is a real $n$-dimensional vector bundle over the base
space $X$ with classifying map $f:X\rightarrow BO(n)$, that is, $V$ is the pull-back of the universal  $n$-dimensional vector bundle over
$BO(n)$ via $f$, then the $i$-th Stiefel-Whitney class of $V$ is given by $w_i(V)=f^*(w_i)$.

We will need the following property of these classes (the injectivity is given by \cite[Proposition 29.2]{Bo53}, the image of
characteristic classes in \cite[9.1, 9.2, and 9.6]{BH}):
\begin{thm}\label{cohoclassicdetect}
 Let ${\mathbf K}=O,U$ or $Sp$, $k$ be any ring if  ${\mathbf K}=U$ or $Sp$, a $\Z /2$-algebra if ${\mathbf K}=O$, $d$ as above.
  The
usual inclusion $j:{\mathbf K}(1)^n\rightarrow {\mathbf K}(n)$ induces an injection in cohomology, such that we have
$$Bj^*(x_i)=\sigma _i(t_1,\ldots , t_n)\in H^*(B{\mathbf K}(1)^n)\cong k[t_1,\ldots ,t_n]$$
where $t_i$'s have degree $d$, and $\sigma _i $ denotes the $i$-th elementary symmetric polynomial.
\end{thm}

\opt{sc}{
\subsection{Central extension of Lie groups}
Let $\alpha: A\rightarrow H\rightarrow G$ be a central extension of Lie groups with $A$ finite.  That is,
$A$ lies in the centre of $H$ with an isomorphism $H/A\cong G$.  Then we get a principal fibration
$BA\rightarrow BH \rightarrow BG$, which is classified by a map $BG\rightarrow K(A,2)$, which corresponds to
a cohomology class $f(\alpha)\in H^2(BG;A)$.  According to \cite[Chapter IV, Lemma 1.12]{AM},
$f$ is well-defined and bijective when $G$ is finite.  It is easy to see that for $G$ Lie groups, $f$ is still
well-defined.  Unfortunately, the proof of bijectivity given there doesn't generalize to the Lie group case as is.
However, according to \cite[Theorem 4]{Wig}, $H^*(BG;A)$ is isomorphic to the ``Borel cohomology'' of
$G$ with coefficients in $A$, and \cite{MCC} identifies the second Borel cohomology with the set of central extension.
It is now easy to see that our map $f$ coincides with the composition of the two bijections.  Thus we get:

\begin{prop}
 Denote $E(A,G)$ be the set of isomorphism classes of central extension of $G$ by $A$.  Then the above construction gives
a well-defined bijection $f_G:E(A,G)\rightarrow H^2(BG;A)$.
\end{prop}

An immediate corollary is the following:
\begin{cor}\label{pullbackext}
Let $\alpha _i: A\rightarrow H_i\rightarrow G_i$ be a central extension of Lie groups with $A$ finite
for $i=1,2$, $\varphi :G_1
\rightarrow G_2$ be a Lie group homomorphism. There is a Lie group homomorphism
${\varphi}^{\prime}:H_1\rightarrow H_2$ that makes the following diagram
commutative  if and only if $\varphi ^*(f_{G_2}(\alpha _2))=\alpha _1$.
$$
\begin{diagram}
 \node{H_1}\arrow{e}\arrow{s,r}{\varphi} \node{G_1}\arrow{s,r}{\varphi ^{\prime}}\\
\node{H_2}\arrow{e} \node{G_2}
\end{diagram}
$$
\end{cor}
}
\subsection{\textcolor{black}{$Pin$ groups, $Pin$-bundles, and $Pin$-structures}}\label{pin}
The orthogonal group $O(n)$ admits several double covers, notably we have central extensions
$\Z /2 \rightarrow Pin^+(n)\rightarrow O(n)$ corresponding to $w_2$ and $\Z /2 \rightarrow Pin^-(n)\rightarrow O(n)$
corresponding to $w_2+w_1^2$ in $H^2(BO(n);\Z /2)$. Similarly the special
orthogonal group $SO(n)$ admits a central extension $\Z /2 \rightarrow Spin(n)\rightarrow SO(n)$ corresponding to $w_2$   (\cite[p.434]{KT}).
These groups can also be defined directly using
Clifford algebras \cite{ABS,KTlow,Lam}.

Given a real vector bundle $V$ over $X$, one can ask whether the structure map can be lifted through the {\it canonical projection} $Pin^{\pm}(n)\rightarrow O(n)$.  Such a lift is called $Pin^{\pm}(n)$-bundle structure.  $V$ admits a $Pin^+$ ( $Pin^-$ respectively) structure if and only if $w_2(V)$ ($w_2(V)+w_1(V)^2$ resp.) vanishes (\cite[Lemma 1.3]{KTlow}).  For a $n$-dimensional manifold $M$, we say that $M$ admits a $Pin^{\pm}(n)$ structure if its tangent bundle admits a $Pin^{\pm}(n)$ structure.  Here we note that this is about a factorisation through particular maps $Pin^{\pm}(n)\rightarrow O(n)$.  Thus although as abstract Lie groups, $Pin^+(4n)$ and $Pin^-(4n)$ are isomorphic (c.f. \cite[example 3 in 1.7, pp. 25-27]{CBDWM}{, communicated to us by Theo Johnson-Freyd,)} where they are called $Pin(4n,0)$ and $Pin(0,4n)$), they are not isomorphic as double covers of $O(4n)$, thus the notion of $Pin^+(4n)$ bundle structure and that of $Pin^-(4n)$ structure don't agree.

The following is well-known (e.g. \cite{KT} p.434):
\begin{prop}\label{rpispin}
 $\R P^{4k}$ has a $Pin^{+}$ structure and $\R P^{4k+2}$ has a $Pin^{-}$ structure.
\end{prop}
{
The proof is left as an exercice to the interested reader.  One can use, for example, the relationship between
the tangent bundle and the canonical line bundle c. f. \cite[Chapter 2, Example 4.8]{Hu}.}
\section{Homology of infinite loop spaces}
For an infinite loop space $X$, since $X\simeq\Omega^2X_2$, the homology $H_*(X;\Z/p)$ is a graded
commutative ring under the Pontryagin product. Moreover, there are Kudo-Araki-Dyer-Lashof homology operations that
we will call Dyer-Lashof operations, $\beta ^{\epsilon}Q^i$ which act on $H_*(X;\Z/p)$. The
operation $Q^i$ is a group homomorphism, is natural with respect to infinite loop maps
\cite[Theorem 1.1]{Maybook}. 
These operations satisfy Adem relations, various Cartan formulae, and Nishida relations \cite[Theorem 1.1]{Maybook}. The algebra
wherein these operations live is the Dyer-Lashof algebra $R$; it is the free algebra generated by these operations, modulo Adem
relations and excess relations. The homology of $H_*(X;\Z/p)$ then becomes an $R$-module. In some cases, these operations allow a neat
description of $H_*(X;\Z/p)$. For instance, if $X=QY$ with $Y$ some path connected space,
then $H_*(X;\Z/p)$ is a free algebra generated by Dyer-Lashof operations on $\widetilde{H}_* (Y;\Z/p)$ (\cite[{Chapter 1,} Lemma 4.10]{Maybook}). Furthermore, when $Y=S^0$, the Dyer-Lashof operations act on the fundamental class of $\widetilde{H}^0(S^0
;\Z /p)$ in such a way that $\{\beta ^{\epsilon}Q^i[1];i\in N, \epsilon =0,1\}$ is precisely the image of $H_*(B\Sigma _p)$ by the
``standard inclusion'' $H_*(B\Sigma _p;\Z /p)\to H_*(QS^0; \Z/p)$.  This latter also coincides with the map induced by the
adjoint of the transfer associated to the inclusion of the trivial group in $\Sigma _p$.

The other cases that we shall consider in this paper, are the spaces $\Z\times BO$ and $\Z\times BU$ which are infinite loop spaces
under Bott periodicity; the monoid structure coming from the Whitney sum is compatible with Bott periodicity.
They correspond to ring spectra $KO$ and $KU$, thus there are maps $QS^0\to \Z \times BO$ and $QS^0\to \Z \times BU$.

We have notably
$$
\begin{array}{lllll}
H^*(BO;\Z/2) &\cong & \Z /2[w_1,w_2, \ldots w_n, \ldots ] & \cong & \varprojlim H^*(BO(n);\Z/2)\\
H_*(BO;\Z/2) &\cong & \Z /2[a_1,a_2, \ldots a_n, \ldots ] & \cong & \varinjlim H_*(BO(n);\Z/2)\\
H^*(BU;\Z) &\cong & \Z [c_1,c_2, \ldots c_n, \ldots ] & \cong & \varprojlim H^*(BU(n);\Z)\\
H_*(BU;\Z) &\cong & \Z [b_1,b_2, \ldots b_n, \ldots ] & \cong & \varinjlim H_*(BU(n);\Z)
\end{array}
$$
The elements $w_i$'s and $c_i$'s are as in Theorem \ref{cohoclassic}.
As $BO$ classifies stable virtual bundles, this means that we can define the Stiefel-Whitney class
$w_i(V)$ for a stable virtual bundle over $X$ with classifying map $f:X\rightarrow BO$ by
$w_i(V)=f^*(w_i)$.  We note that the multiplication by $(-1)$ on the set of virtual bundles corresponds to
the ``multiplication by $(-1)$'' self-map on $BO$, thus the conjugation $\tau$ on $H^*(BO;\Z/2)$ satisfies
$f^*\tau(w_i)=w_i(-V)$.  Similar statements hold for $BU$. {As to the homology is concerned, we will use the fact that the elements
$a_i$'s and $b_i$'s are respectively the image of a generator of $H_i(BO(1);\Z/2)$ and $H_{2i}(BU(1);\Z)$.}\\

The map induced in homology by the unit map was determined in \cite{Priddyoperations}, in particular, we have 
\begin{lmm}\label{QS-BO} \cite[Proposition 4.10, $n=1$ case]{Priddyoperations}
The map $H_*(QS^0;\Z/2)\lra H_*(\Z\times BO;\Z/2)$ is an epimorphism.
\end{lmm}
Basically this is because $H_*(BO;\Z/2)$ is generated by $H_*(BO(1);\Z/2)$,  $H_*(BO(1);\Z/2)$ is ``contained''
in $H_*(Q_0S^0;\Z/2)$, and the inclusions $H_*(BO(1);\Z/2)\subset H_*(\Z\times BO;\Z/2)$ and $H_*(BO(1);\Z/2)\subset
H_*(Q_0S^0;\Z/2)$ are compatible.

Now we would like to generalize to the ``comlex'' case.  Although it is still true that the homology of $BU$ is generated by
that of $BU(1)$, for odd prime $p$ the homology of $BU(1)$ contains elements that are unrelated to the Dyer-Lashof operations.
That is, the homology of $BU(1)$ is the even degree part of the homology of $B\Z/p$, which is much
lagrer than that of $B\Sigma _p$ for an odd prime $p$, which is related to the Dyer-Lashof operations.  However,
as is well-known, (stably) $B\Sigma _p$ splits off $BZ/p$, and it turns out that we can also split $BU(1)$ and $BU$ in
a compatible way, which
which we will discuss below.

Denote by $KU$ the complex $K$-theory spectrum.  \opt{local}{After localising at $p$, }$KU$ splits as $KU\simeq \vee _{i=0}^{p-2}\Sigma ^{2i}E(1)$, where $\pi_*(E(1))\cong \Z _{(p)}[v_1,v_1^{-1}]$ with degree of $v_1$ equal to $2(p-1)$ \cite[Lecture 4]{A}. Denote $j_E$ the resulting splitting map $E(1)\rightarrow KU$. Similarly we have a splitting $\C P^{\infty }_{\opt{local}{(p)}}\cong \vee _{i=0}^{p-2}X(i)$ with $H^*(X(i))=0$ unless $*\equiv  2i.\bmod 2(p-1)$\cite{MNT}.

Since $BU\times \Z_{(p)}$ is the infinite loop space associated to $KU$, it also splits as a product of spaces
$${BU\times\Z_{(p)}\simeq \Pi _{i=0}^{p-2}\Omega ^{\infty}\Sigma ^{2i}E(1).}$$
Thus $\Omega ^{\infty}E(1)$ is a \textcolor{black}{decomposition factor} of $BU\times\Z_{(p)}$.  Its cohomology can be described as follows:

\begin{lmm}Let $k=\Z _{(p)},\Q$ or $\Z /p$.\label{cohe1}
 $H^*(\Omega ^{\infty}E(1),k)\cong k[c_{p-1},c_{2(p-1)},\ldots c_{m(p-1)},\ldots]$, and $j_E^*$ sends $c_{m(p-1)}\in
H^{2m(p-1)}(BU;k)$ to $c_{m(p-1)}$ and other $c_i$'s to $0$.
\end{lmm}
\begin{proof}
This can be shown using \cite{HH}, but here we follow rather the arguments in \cite{HR}.
Let's start with the case $k=\Q$.  For $k$-vector spaces $V$, denote by $Sym_k(V)$ the symmetric algebra generated by $V$,
i.e., $\oplus _q V^{\otimes _kq}/\Sigma _q$ with the product induced by the concatenation.
\opt{sc}{By Proposition \ref{hinfrational}, we see} \opt{short}{It is well known} that for any spectra $X$ with
$\pi _{odd}(X)=0$, we have natural isomorphisms
$$H^*(\Omega ^{\infty}_0X;\Q)\cong H^*((\Omega ^{\infty}_0X)_{\Q};\Q)\cong H^*(\Omega ^{\infty}_0(X_{\Q});\Q)\cong
Sym _{\Q}Hom(\pi _{*>0}(X);\Q ).$$
Since $\pi _*(j_E)$ is bijective for $*=2m(p-1)$ and $0$ otherwise, we get the desired result in
this case.
As $\Omega ^{\infty}E(1)$ is a \textcolor{black}{direct factor} of $BU\times \Z_{(p)}$,
$H_*(\Omega ^{\infty}E(1);\Z_{(p)})$ is torsion-free.  Therefore $H_*(\Omega ^{\infty}E(1);\Z_{(p)})$ injects to
$H_*(\Omega ^{\infty}E(1);\Q )$.  Similarly for $BU\times \Z_{(p)}$.  Thus we get the result when $k=\Z_{(p)}$. Finally, one can derive the case $k=\Z /p$ follows from this by the universal coefficient theorem.
\end{proof}
There is another splitting involving the spaces related to them. Consider the orientation map for the $KU$-theory $\C P^{\infty }\to BU$.
Since the target is an infinite loop space, by the adjointness, it factors through $Q\C P^{\infty }$.    Then the map
$Q\C P^{\infty }\to BU$ splits as a map of spaces, that is, there is a space $F$ such that $Q\C P^{\infty }
\simeq BU \times F$ \cite[Theorem]{Seg}.  It turns out that the {Adams'} splittings of $KU$ and $\C P^{\infty }$ are compatible, so that Segal's splitting can be refined to the splitting of corresponding Adams' pieces \cite[Theorem 1.1]{Kono}. We have
\begin{prop}\label{splitkono}
 The map $\Omega ^{\infty}X(0)\to  \Omega ^{\infty}E(1)$ splits, that is we have a space $F^{\prime}$ such that
$\Omega ^{\infty}X(0)\simeq  \Omega ^{\infty}E(1)\times F^{\prime}$.  In particular, it induces a surjection in homology with
any coefficient.
\end{prop}
Now we are ready to prove the following.
\begin{lmm}
The unit map of the ring spectrum $E(1)$ induces a surjection $H_*(QS^0;\Z /p)\to H_*(\Omega ^{\infty}E(1);\Z /p)$.
\end{lmm}
\begin{proof}
By Proposition \ref{splitkono} the map $H_*(\Omega ^{\infty}X(0);\Z /p)\to H_*(\Omega ^{\infty}E(1);\Z /p)$ is surjective. However, as $X(0)$ is a retract of a space $\C P^{\infty }$,  $H_*(\Omega ^{\infty}X(0);\Z /p)$ is generated by elements of $H_*(X(0);\Z /p)$ under the Pontrjagin product and Dyer-Lashof operations.  But  $H_*(X(0);\Z /p)$ is the even degree part of $H_*(B\Sigma _p;\Z/p)$, that is we have a following commutative diagram
$$\begin{diagram}
\node{B\Z/p}\arrow{e}\arrow{s} \node{\C P^{\infty}}\arrow{s}\\
\node{B\Sigma _p}\arrow{e} \node{X(0)}
  \end{diagram}
$$
so that the horizontal arrows kill $H_{odd}(-;\Z/p)$ and induce isomorphisms on $H_{even}(-;\Z/p)$. So the  image of $H_*(X(0);\Z/p)$ in $H_*(\Omega ^{\infty}E(1);\Z /p)$ is just Dyer-Lashof operations applied to the image of the fundamental class of $H_*(QS^0;\Z /p)$ in $H_*(\Omega ^{\infty}E(1);\Z /p)$, up to translation by connected component. Thus $H_*(\Omega ^{\infty}E(1);\Z /p)$ is generated by the image of the fundamental class of $H_*(QS^0;\Z /p)$ under the Pontrjagin product and Dyer-Lashof operations.  In other words, the map $H_*(QS^0;\Z /p)\to H_*(\Omega ^{\infty}E(1);\Z /p)$ is surjective.

Alternatively, this can be deduced from \cite[Proposition 4.6]{Priddyoperations} and Proposition \ref{cohe1}.
\end{proof}

\tb{Acknowledgements.}
We are grateful to Haynes Miller for helpful correspondences and to Oscar Randal-Williams for helpful correspondences and his
comments on earlier versions of this paper. We also thank the referees of various journals for their critical readings. The first author
thanks IPM and University of Tehran for their hospitality.  The second author thanks Institut Fourier for its hospitality during {visits
during June 2013 and October 2014}.

\end{document}